\documentclass[a4paper,11pt]{article}
\usepackage[utf8]{inputenc}
\usepackage[T1]{fontenc}
\usepackage{textcomp}
\usepackage{fancyhdr}
\usepackage{fullpage}
\usepackage{lmodern}
\usepackage{amsmath,amssymb,bm,bbm}
\usepackage{algorithm}
\usepackage{algorithmic}
\usepackage{float}
\usepackage{graphicx}
\usepackage{extarrows}
\usepackage{caption}
\usepackage{graphicx, subfig}
\usepackage{titling}
\usepackage{yhmath}
\usepackage{mathrsfs}
\usepackage{cite}
\usepackage{footnote}
\usepackage{amsthm}
\usepackage{color}
\usepackage{slashed}
\usepackage{verbatim}
\usepackage[colorlinks=true,linkcolor=red,citecolor=blue]{hyperref}
\usepackage{amsfonts,euscript}
\usepackage{latexsym, multicol}
\usepackage{epstopdf}
\usepackage{psfrag}
\usepackage{mathrsfs}
\usepackage{xcolor}
\usepackage{comment}
\usepackage{authblk}
\usepackage{indentfirst}
\usepackage{lipsum}
\DeclareFontFamily{U}{mathx}{\hyphenchar\font45}
\DeclareFontShape{U}{mathx}{m}{n}{
<5> <6> <7> <8> <9> <10>
<10.95> <12> <14.4> <17.28> <20.74> <24.88>
mathx10}{}
\DeclareSymbolFont{mathx}{U}{mathx}{m}{n}
\DeclareFontSubstitution{U}{mathx}{m}{n}
\DeclareMathAccent{\widecheck}{0}{mathx}{"71}
\numberwithin{equation}{section}
\allowdisplaybreaks[3]
\renewcommand{\O}{\mathbb{O}}
\def\Gaw{\Ga_w}
\def\sic{\wideparen{\si}}
\def\F{\mathbb{E}}
\def\Fb{\mathbb{F}}

\def\vkp{\varkappa}
\renewcommand{\c}{\cdot}

\newcommand{\unl}{\underline{L}}

\newcommand{\pr}{\partial}
\newcommand{\ub}{{\underline{u}}}
\newcommand{\unc}{{\underline{C}}}
\newcommand{\bb}{{\underline{\beta}}}
\newcommand{\bg}{\mathbf{g}}
\newcommand{\mo}{\mathcal{O}}
\newcommand{\chib}{\underline{\chi}}
\newcommand{\omb}{\underline{\omega}}
\def\etab{{\underline{\eta}}}
\def\ee{\mathcal{N}}
\def\mum{[\mu]}
\def\mumb{[\mub]}

\def\sfr{\mathfrak{s}}
\def\sk{\sfr}

\newcommand{\kk}{\mathcal{K}}
\newcommand{\mr}{\mathcal{R}}
\newcommand{\ur}{\underline{\mr}}

\newcommand{\M}{\mathcal{M}}
\newcommand{\D}{\mathbf{D}}
\newcommand{\uf}{\underline{F}}
\newcommand{\cuv}{{C_u^V}}
\newcommand{\ucuv}{{\unc_\ub^V}}
\newcommand{\dd}{{\mathfrak{d}}}

\newcommand{\Ric}{{\ric}}
\DeclareMathOperator{\sRic}{Ric}
\newcommand{\g}{\bg}
\newcommand{\R}{{\mathbf{R}}}

\def\mumc{\widecheck{\mum}}
\def\mumbc{\widecheck{\mumb}}

\def\hot{\widehat{\otimes}}

\def\II{\mathcal{I}}
\def\la{\lambda}
\def\rhoc{\wideparen{\rho}}

\def\fc{\widecheck{f}}
\def\MM{\mathcal{M}}
\def\Cb{\unc}
\def\om{\omega}
\def\ze{\zeta}
\def\Om{\Omega}
\def\aa{{\underline{\a}}}
\def\Si{\Sigma}
\def\si{\sigma}
\def\mub{{\underline{\mu}}}
\def\dk{\dd}
\def\dkb{\slashed{\dk}}
\def\ga{\gamma}
\def\Ga{\Gamma}
\def\xib{\underline{\xi}}
\def\a{\alpha}
\def\b{\beta}
\def\hch{\widehat{\chi}}
\def\hchb{\widehat{\chib}}
\def\trch{\tr\chi}
\def\trchb{\tr\chib}
\def\trchc{\widecheck{\trch}}
\def\trchbc{\widecheck{\trchb}}

\def\de{\delta}
\def\De{\Delta}
\def\nab{\nabla}
\def\ov{\overline}

\def\bbb{\underline{b}}
\def\Gab{\Ga_b}
\def\Gag{\Ga_g}
\def\ep{\varepsilon}
\def\les{\lesssim}
\def\RR{\mr}

\DeclareMathOperator{\curl}{curl}
\DeclareMathOperator{\grad}{grad}

\DeclareMathOperator{\tr}{tr}
\DeclareMathOperator{\sdiv}{div}
\def\bdiv{\mathbf{Div}}
\def\ric{\mathbf{Ric}}
\newtheorem{thm}{Theorem}[section]
\newtheorem{prop}[thm]{Proposition}
\newtheorem{lem}[thm]{Lemma}

\newtheorem{rk}[thm]{Remark}
\newtheorem{df}[thm]{Definition}
\title{Exterior stability of Minkowski spacetime with borderline decay}
\author{Dawei Shen\footnote{Email address: ds4350@columbia.edu \par\indent\hspace{0.256cm} Department of Mathematics, Columbia University, New York, NY, 10027.}}

\begin{document}
\maketitle
\begin{abstract}
\noindent\textbf{Abstract.}
In 1993, the global stability of Minkowski spacetime was proved in the celebrated work of Christodoulou and Klainerman. In 2003, Klainerman and Nicol\`o revisited Minkowski stability in the exterior of an outgoing null cone. In 2023, the author extended the results of Christodoulou-Klainerman to minimal decay assumptions. In this paper, we prove that the exterior stability of Minkowski holds with decay that is borderline compared to the minimal decay considered in 2023.

\end{abstract}

\bigskip
\begingroup
\small
\begin{center}
{\large Stabilit\'e ext\'erieure de l'espace-temps de Minkowski avec une d\'ecroissance limite}
\end{center}

\quotation
\noindent\textbf{R\'esum\'e.}
En 1993, la stabilit\'e globale de l'espace-temps de Minkowski a \'et\'e prouv\'ee dans les c\'el\`ebres travaux de Christodoulou et Klainerman. En 2003, Klainerman et Nicol\`o ont revisit\'e la stabilit\'e de Minkowski \`a l'ext\'erieur d'un c\^one de lumi\`ere sortant. En 2023, l'auteur a \'etendu les r\'esultats de Christodoulou et Klainerman aux hypoth\`eses de d\'ecroissance minimale. Dans cet article, nous prouvons que la stabilit\'e ext\'erieure de Minkowski tient avec une d\'ecroissance qui est limite par rapport \`a la d\'ecroissance minimale consid\'er\'ee par l'auteur en 2023.

\medskip
\noindent\textbf{Mots-cl\'es:}
stabilit\'e de Minkowski, feuilletage double-nul, d\'ecroissance limite, \'estimations de $r^p$ \`a poids.
\endquotation
\endgroup
\section{Introduction}
\subsection{Einstein vacuum equations and the Cauchy problem}
A Lorentzian $4$--manifold $(\MM,\g)$ is called a vacuum spacetime if it solves the Einstein vacuum equations:
\begin{equation}\label{EVE}
    \Ric(\g)=0\quad\;\; \mbox{ in }\;\MM,
\end{equation}
where $\Ric$ denotes the Ricci tensor of the Lorentzian metric $\g$. The Einstein vacuum
equations are invariant under diffeomorphisms, and therefore one considers equivalence
classes of solutions. Expressed in general coordinates, \eqref{EVE} is a non-linear geometric coupled system of partial differential equations of order 2 for $\g$. In suitable coordinates, for example so-called wave coordinates, it can be shown that \eqref{EVE} is hyperbolic and hence admits an initial value formulation.\\ \\
The corresponding initial data for the Einstein vacuum equations is given by specifying a
triplet $(\Si,g,k)$ where $(\Si,g)$ is a Riemannian $3$--manifold and $k$ is the traceless symmetric $2$--tensor on $\Si$ satisfying the constraint equations:
\begin{align}
\begin{split}\label{constraintk}
R&=|k|^2-(\tr k)^2,\\
D^jk_{ij}&=D_i(\tr k),
\end{split}
\end{align}
where $R$ denotes the scalar curvature of $g$, $D$ denotes the Levi-Civita connection of $g$ and
\begin{align*}
    |k|^2:=g^{ad}g^{bc}k_{ab}k_{cd},\qquad\quad\tr k:=g^{ij}k_{ij}.
\end{align*}
In the future development $(\MM,\g)$ of such initial data $(\Si,g,k)$, $\Si\subset \MM$ is a spacelike hypersurface with an induced metric $g$ and a second fundamental form $k$.\\ \\
The seminal well-posedness results for the Cauchy problem obtained in \cite{cb,cbg} ensure that for any smooth Cauchy data, there exists a unique smooth maximal globally hyperbolic development $(\MM,\g)$ solution of Einstein equations \eqref{EVE} such that $\Si\subset \MM$ and $g$, $k$ are respectively the first and second fundamental forms of $\Si$ in $\MM$. \\ \\
The prime example of a vacuum spacetime is Minkowski spacetime:
\begin{equation*}
    \MM=\mathbb{R}^4,\qquad \g=-dt^2+(dx^1)^2 +(dx^2)^2+(dx^3)^2,
\end{equation*}
for which Cauchy data are given by
\begin{equation*}
    \Si=\mathbb{R}^3,\qquad g=(dx^1)^2+(dx^2)^2+(dx^3)^2,\qquad k=0.
\end{equation*}
In the present work, we consider the problem of the stability of Minkowski spacetime and start with reviewing the literature on this problem.
\subsection{Previous works of the stability of Minkowski spacetime}
In 1993, Christodoulou and Klainerman \cite{Ch-Kl} proved the global stability of Minkowski for the Einstein-vacuum equations, a milestone in the domain of mathematical general relativity. In 2003, Klainerman and Nicol\`o \cite{Kl-Ni} proved the Minkowski stability in the exterior of an outgoing cone using the double null foliation. Moreover, Klainerman and Nicol\`o \cite{knpeeling} showed that under stronger asymptotic decay and regularity properties than those used in \cite{Ch-Kl,Kl-Ni}, asymptotically flat initial data sets lead to solutions of the Einstein vacuum equations which have strong peeling properties. Given that the goal of this paper is to extend the result of \cite{Kl-Ni}, we will state the results of \cite{Ch-Kl,Kl-Ni} in Section \ref{ssec6.2}.\\ \\
We now mention other proofs of Minkowski stability. In 2007, Bieri \cite{Bieri} gave a new proof of global stability of Minkowski requiring one less derivative and fewer vectorfields compared to \cite{Ch-Kl}. Lindblad and Rodnianski \cite{lr1,lr2} gave a new proof of the stability of the Minkowski spacetime using \emph{wave-coordinates} and showing that the Einstein equations verify the so called \emph{weak null structure} in that gauge. Huneau \cite{huneau} proved the nonlinear stability of Minkowski spacetime with a translation Killing field using generalized wave-coordinates. Using the framework of Melrose’s b-analysis, Hintz and Vasy \cite{hv} reproved the stability of Minkowski space. Graf \cite{graf} proved the global nonlinear stability of Minkowski space in the context of the spacelike-characteristic Cauchy problem for Einstein vacuum equations, which together with \cite{Kl-Ni} allows reobtaining \cite{Ch-Kl}. Under the framework of \cite{Kl-Ni} and using $r^p$--weighted estimates of Dafermos and Rodnianski \cite{Da-Ro}, the author \cite{Shen22} reproved the Minkowski stability in exterior regions. More recently, Hintz \cite{Hintz} reproved the Minkowski stability in exterior regions by using the framework of \cite{hv}. Using $r^p$--weighted estimates, the author \cite{Shen23} extends the results of \cite{Bieri} to minimal decay assumption.\\ \\
There are also stability results concerning Einstein's equations coupled with non trivial matter fields, see for example the introduction of \cite{Shen22,Shen23}.
\subsection{Minkowski Stability in \texorpdfstring{\cite{Ch-Kl,Kl-Ni}}{}}\label{ssec6.2}
We recall in this section the results in \cite{Ch-Kl,Kl-Ni}. First, we recall the definition of a \emph{maximal hypersurface}, which plays an important role in the statements of the main theorems in \cite{Ch-Kl,Kl-Ni}.
\begin{df}\label{def6.1}
An initial data $(\Si,g,k)$ is posed on a maximal hypersurface if it satisfies 
\begin{equation}
    \tr k=0.
\end{equation}
In this case, we say that $(\Si,g,k)$ is a maximal initial data set, and the constraint equations \eqref{constraintk} reduce to
\begin{equation}
    R=|k|^2,\qquad \sdiv k=0,\qquad \tr k=0.
\end{equation}
\end{df}
We introduce the notion of \emph{$(s,q)$--asymptotically flat initial data}.
\begin{df}\label{defasym}
Given $s\geq 1$ and $q\geq 0$, we say that a data set $(\Sigma_0,g,k)$ is $(s,q)$--asymptotically flat if there exists a coordinate system $(x^1,x^2,x^3)$ defined outside a sufficiently large compact set such that:
\begin{itemize}
\item In the case of $s\geq 3$
\begin{align}
    \begin{split}
    g_{ij}&=\left(1-\frac{2M}{r}\right)^{-1}dr^2+r^2 d\si_{\mathbb{S}^2}+o_{q+1}(r^{-\frac{s-1}{2}}),\qquad k_{ij}=o_q(r^{-\frac{s+1}{2}}).
    \end{split}
\end{align}
\item In the case of $1\leq s<3$
\begin{align}
    \begin{split}
        g_{ij}&=\de_{ij}+o_{q+1}(r^{-\frac{s-1}{2}}),\qquad k_{ij}=o_q(r^{-\frac{s+1}{2}}).
    \end{split}
\end{align}
\end{itemize}
Here, the notation $f=o_l(r^{-m})$ means that $\pr^\a f=o(r^{-m-|\a|})$ holds for all $|\a|\leq l$.
\end{df}
We also introduce the following functional:
\begin{align*}
    J_0(\Si_0,g,k):=\sup_{\Si_0}\left((d_0^2+1)^3 |\sRic|^2 \right)+\int_{\Si_0}\sum_{l=0}^3(d_0^2+1)^{l+1}|D^l k|^2+\int_{\Si_0}\sum_{l=0}^1 (d_0^2+1)^{l+3}|D^l B|^2,
\end{align*}
where $d_0$ is the geodesic distance from a fixed point $O\in\Si_0$, and $B_{ij}:=(\curl\widehat{{R}})_{ij}$ is the \emph{Bach tensor}, $\widehat{{R}}$ is the traceless part of $\sRic$. Now, we can state the main theorems of \cite{Ch-Kl} and \cite{Kl-Ni}.
\begin{thm}[Global stability of Minkowski space \cite{Ch-Kl}]\label{ckmain}
There exists an $\ep_0>0$ sufficiently small such that if $J_0(\Si_0,g,k)\leq\ep_0^2$, then the initial data set $(\Si_0,g,k)$, $(4,3)$--asymptotically flat (in the sense of Definition \ref{defasym}) and maximal, has a unique, globally hyperbolic, smooth, geodesically complete solution. This development is globally asymptotically flat, i.e. the Riemann curvature tensor tends to zero along any causal or spacelike geodesic. Moreover, there exists a global maximal time function $t$ and an optical function $u$, which satisfies $\g^{\a\b}\pr_\a u\pr_\b u=0$, defined everywhere in an external region.
\end{thm}
\begin{thm}[Exterior stability of Minkowski \cite{Kl-Ni}]\label{knmain}
Consider an initial data set $(\Si_0,g,k)$, $(4,3)$--asymptotically flat and maximal, and assume $J_0(\Si_0,g,k)$ is bounded. Then, given a sufficiently large compact set $K\subset\Si_0$ such that $\Sigma_0 \setminus K$ is diffeomorphic to $\mathbb{R}^3\setminus\overline{B}_1$, and under additional smallness assumptions, there exists a unique future development $(\M,\g)$ of $\Si_0\setminus K$ with the following properties:
\begin{itemize}
    \item $(\M,\bg)$ can be foliated by a double null foliation $(C_u,\Cb_\ub)$ whose outgoing leaves $C_u$ are complete.
    \item We have detailed control of all the quantities associated with the double null foliations of the spacetime, see Theorem 3.7.1 in \cite{Kl-Ni}.
\end{itemize}
\end{thm}
\begin{rk}
Various choices of the parameter $s$ in Definition \ref{defasym} have been made in the literature. In particular:
\begin{itemize}
    \item $s=4$ is used for example in \cite{Ch-Kl,Kl-Ni,graf}.
    \item $s=3+\de$ is used for example in \cite{lr2,hv,Hintz}.
    \item $s=3-\de$ is used in \cite{ionescu} for the Einstein-Klein-Gordon system.
    \item $s\in(2,3]$ is used in \cite{leflochMa} for the Einstein-Klein-Gordon system.
    \item $s=2$ is used in \cite{Bieri}.
    \item $s\in(1,2]$ is used in \cite{Shen23}.
\end{itemize}
Here, $\de$ denotes a constant satisfying $0<\de\ll 1$. Notice that this list focuses on low values of $s$. For large values of $s$, note, for example, \cite{knpeeling} valid for $s>7$ and \cite{Shen22} valid for $s>3$.
\end{rk}
The goal of this paper is to extend the results of \cite{Kl-Ni,Shen22} to $s=1$ which is borderline compared to the range considered in \cite{Shen23}. More precisely, we prove the exterior stability of Minkowski spacetime for $(1,3)$--asymptotically flat initial data.
\subsection{Rough version of the main theorem}
In this section, we state a simple version of our main theorem. For an explicit statement, see Theorem \ref{maintheorem}.
\begin{thm}[Main theorem (first version)]
Let $(\Si_0,g,k)$ be an initial data set which is $(1,3)$--asymptotically flat in the sense of Definition \ref{defasym}. Let $K\subset\Si_0$ be a sufficiently large compact set such that $\Si_0\setminus K$ is diffeomorphic to $\mathbb{R}^3\setminus\ov{B}_1$. Assume that we have a smallness condition on $\Si_0\setminus K$. Then, there exists a unique future development $(\M,\g)$ of $\Si_0\setminus K$ with the following properties:
\begin{itemize}
\item $(\M,\g)$ can be foliated by a double null foliation $(C_u,\unc_\ub)$ whose outgoing leaves $C_u$ are complete for all $u\leq u_0$.
\item We have detailed control of all the quantities associated with the double null foliations of spacetime; see Theorem \ref{maintheorem}.
\end{itemize}
\end{thm}
\subsection{Structure of the paper}
\begin{itemize}
    \item In Section \ref{sec7}, we recall the fundamental notions and the basic equations.
    \item In Section \ref{sec8}, we present the main theorem. We then state intermediate results, and prove the main theorem. The rest of the paper focuses on the proof of these intermediary results.
    \item In Section \ref{secboot}, we make bootstrap assumptions and prove first consequences. These consequences will be used frequently in the rest of the paper.
    \item In Section \ref{sec9}, we apply $r^p$--weighted estimates to Bianchi equations to control the flux of curvature.
    \item In Section \ref{sec10}, we estimate the $L^p(S)$--norms of curvature components and Ricci coefficients using the null structure equations and Bianchi equations.
\end{itemize}
\subsection{Acknowledgements} The author would like to thank Sergiu Klainerman and J\'er\'emie Szeftel for their support, discussions, and encouragement.
\section{Preliminaries}\label{sec7}
\subsection{Geometry set-up}\label{doublenullfoliation}
Let $\Si_0$ be a spacelike hypersurface. Let $K\subset\Si_0$ be a compact subset such that $\Si_0\setminus K$ is diffeomorphic to $\mathbb{R}^3 \setminus \ov{B_1}$ where $\ov{B_1}$ is the unit closed ball in $\mathbb{R}^3$. We fix a radial foliation on the initial hypersurface $\Sigma_0 \setminus K$ by the level sets of a scalar function $w$. The leaves are denoted by
\begin{equation}
    S_{(0)}(w_1)=\{p\in\Si_0/\, w(p)=w_1\},
\end{equation}
where $w_1\in \mathbb{R}$. We assume that
\begin{equation}
    \pr K=\{p\in\Si_0/\,w(p)=w_0\},\qquad K=\{p\in\Si_0/\, w(p)\leq w_0\},
\end{equation}
where $w_0$ is the area radius of $\pr K$ defined by
\begin{equation}\label{defw0}
    w_0:=\sqrt{\frac{|\pr K|}{4\pi}}.
\end{equation}
The \emph{double null foliation} was first introduced by Klainerman and Nicol\`o in \cite{Kl-Ni} in order to treat the exterior stability of Minkowski. Proceeding as in \cite{Kl-Ni}, we now construct a double null foliation in the future of $\Si_0\setminus K$. For this, we denote $T$ the normal vectorfield of $\Si_0$ oriented towards the future and $N$ the unit vectorfield tangent to $\Si_0$, oriented towards infinity and orthogonal to the leaves $S_0(w)$. We define two null vectors on $\Sigma_0\setminus K$ by
\begin{equation}
    L:=N(w)(T+N),\qquad\quad\unl:=N(w)(T-N).\label{tntn}
\end{equation}
We extend the definition of $L$ and $\unl$ to the future of $\Si_0\setminus K$ by the geodesic equations:
\begin{equation}
    \D_{L}L=0,\qquad \D_{\unl} \unl=0.\label{geo}
\end{equation}
We define a positive function $\Omega$ by the formula
\begin{equation}\label{deflapse}
    \bg(L,\unl)=-\frac{1}{2\Omega^2},
\end{equation}
where $\Om$ is called the \emph{null lapse function}. We then define two optical functions $u$ and $\ub$ in the future of $\Si_0\setminus K$ by the \emph{Eikonal equations}
\begin{equation*}
    L=-\grad u,\qquad\quad \unl=-\grad\ub,
\end{equation*}
and the initial conditions
\begin{equation*}
    u\big|_{\Sigma_0\setminus K}=-w,\qquad\quad\ub\big|_{\Sigma_0\setminus K}=w.
\end{equation*}
We then define the normalized null pair $(e_3,e_4)$ by
\begin{equation}\label{choiceLL}
   e_3=2\unl,\qquad\quad e_4=2\Om^2 L.
\end{equation}
\begin{rk}
    A seemingly more natural null frame to use would be
    \begin{equation}\label{usualchoice}
        e_3=2\Om\unl,\qquad\quad e_4=2\Om L,
    \end{equation}
    see for example \cite{Kl-Ni,Shen22}. The null frame introduced in \eqref{choiceLL} is similar to \cite{Taylor}. The reason for the choice here is related to the fact that $\omb=0$ in this frame, see \eqref{6.6}.
\end{rk}
The spacetime in the future of $\Si_0\setminus K$ is then foliated by the level sets of $u$ and $\ub$ respectively. We use $C_u$ to denote the outgoing null hypersurfaces which are the level sets of $u$ and use $\Cb_\ub$ to denote the incoming null hypersurfaces which are the level sets of $\ub$. We also denote
\begin{equation*}
    S(u,\ub):=C_u\cap\Cb_\ub,
\end{equation*}
which are spacelike $2$--spheres. On a given two sphere $S(u,\ub)$, we choose a local frame $(e_1,e_2)$, we call $(e_1,e_2,e_3,e_4)$ a null frame. As a convention, throughout the paper, we use capital Latin letters $A,B,C,...$ to denote an index from 1 to 2 and Greek letters $\a,\b,\ga,...$ to denote an index from 1 to 4, e.g. $e_A$ denotes either $e_1$ or $e_2$.\\ \\
For any sphere $S(u,\ub)$ of the double null foliation $(u,\ub)$, we denote its causal past by
\begin{align*}
    V(u,\ub):=J^-(S(u,\ub))\cap D^+(\Si_0\setminus K).
\end{align*}
Denoting $u_0=-w_0$ and taking $\ub_*>w_0$, we define the bootstrap region by:
\begin{figure}
  \centering
  \includegraphics[width=0.95\textwidth]{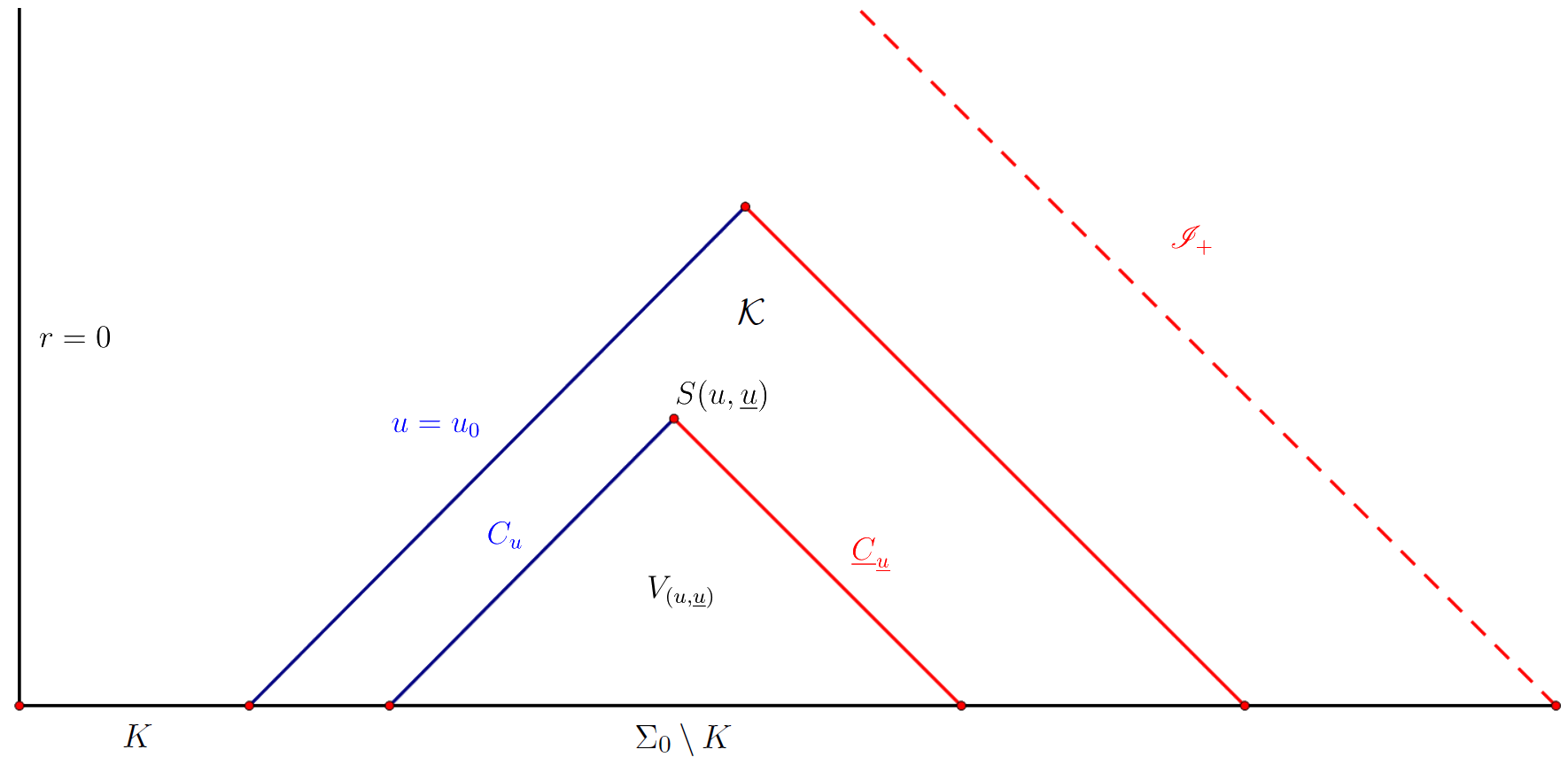}
  \caption{Description of $V(u,\ub)$ and $S(u,\ub)$}\label{fig}
\end{figure}
\begin{align*}
    \kk:=V(u_0,\ub_*).
\end{align*}
See Figure \ref{fig} for a geometric description. By the local existence theorem, the spacetime $\kk$ exists if $\ub_*$ is sufficiently close to $w_0$. We now introduce a local coordinate system $(u,\ub,\phi^A)$ in $\kk$ with $e_4(\phi^A)=0$. In this coordinate system, the spacetime metric $\g$ takes the following form:
\begin{equation}\label{metricg}
\g=-2\Om^2(d\ub\otimes du+du\otimes d\ub)+\ga_{AB}(d\phi^A-\bbb^Adu)\otimes (d\phi^B-\bbb^Bdu),
\end{equation}
where $\ga$ denotes the induced metric on $S(u,\ub)$ and $\bbb:=\bbb^A\pr_{\phi^A}$ denotes the \emph{null shift}. Note that
\begin{equation}\label{pruprub}
   \Om^2e_3=\pr_u+\bbb,\qquad\quad e_4=\pr_{\ub}.
\end{equation}
We recall the null decomposition of the Ricci coefficients and curvature components of the null frame $(e_1,e_2,e_3,e_4)$ as follows:
\begin{align}
\begin{split}\label{defga}
\chib_{AB}&=\g(\D_A e_3, e_B),\qquad\quad \chi_{AB}=\g(\D_A e_4, e_B),\\
\xib_A&=\frac 1 2 \g(\D_3 e_3,e_A),\qquad\quad\,\,  \xi_A=\frac 1 2 \g(\D_4 e_4, e_A),\\
\omb&=\frac 1 4 \g(\D_3e_3 ,e_4),\qquad\quad \,\,\,\;\, \om=\frac 1 4 \g(\D_4 e_4, e_3), \\
\etab_A&=\frac 1 2 \g(\D_4 e_3, e_A),\qquad\quad\;  \eta_A=\frac 1 2 \g(\D_3 e_4, e_A),\\
\ze_A&=\frac 1 2 \g(\D_{e_A}e_4, e_3),
\end{split}
\end{align}
and
\begin{align}
\begin{split}\label{defr}
\a_{AB} &=\R(e_A, e_4, e_B, e_4),\qquad\quad\,\aa_{AB} =\R(e_A, e_3, e_B, e_3), \\
\b_{A} &=\frac 1 2\R(e_A, e_4, e_3, e_4),\qquad\quad\,\,\bb_{A}=\frac 1 2 \R(e_A, e_3, e_3, e_4),\\
\rho&=\frac 1 4 \R(e_3, e_4, e_3, e_4), \qquad\quad\;\;\;\; \si =\frac{1}{4}{^*\R}(e_3, e_4, e_3,e_4),
\end{split}
\end{align}
where $^*\R$ denotes the Hodge dual of $\R$. The null second fundamental forms $\chi, \chib$ are further decomposed in their traces $\trch$ and $\trchb$, and traceless parts $\hch$ and $\hchb$:
\begin{align*}
\trch&:=\de^{AB}\chi_{AB},\qquad\quad \,\hch_{AB}:=\chi_{AB}-\frac{1}{2}\de_{AB}\trch,\\
\trchb&:=\de^{AB}\chib_{AB},\qquad\quad \, \hchb_{AB}:=\chib_{AB}-\frac{1}{2}\de_{AB}\trchb.
\end{align*}
We define the horizontal covariant operator $\nab$ as follows:
\begin{equation*}
\nab_X Y:=\D_X Y-\frac{1}{2}\chib(X,Y)e_4-\frac{1}{2}\chi(X,Y)e_3.
\end{equation*}
We also define $\nab_4 X$ and $\nab_3 X$ to be the horizontal projections:
\begin{align*}
\nab_4 X&:=\D_4 X-\frac{1}{2} \g(X,\D_4e_3)e_4-\frac{1}{2} \g(X,\D_4e_4)e_3,\\
\nab_3 X&:=\D_3 X-\frac{1}{2} \g(X,\D_3e_3)e_3-\frac{1}{2} \g(X,\D_3e_4)e_4.
\end{align*}
A tensor field $\psi$ defined on $\MM$ is called tangent to $S$ if it is a priori defined on the spacetime $\M$ and all the possible contractions of $\psi$ with either $e_3$ or $e_4$ are zero. We use $\nabla_3 \psi$ and $\nabla_4 \psi$ to denote the projection to $S(u,\ub)$ of usual derivatives $\D_3\psi$ and $\D_4\psi$. As a direct consequence of \eqref{defga}, we have the following Ricci formulae:
\begin{align}
\begin{split}\label{ricciformulae}
    \D_A e_3&=\chib_{AB}e_B+\ze_A e_3,\qquad\qquad\qquad\;\,\, \D_A e_4=\chi_{AB}e_B-\ze_A e_4,\\
    \D_3 e_A&=\nab_3 e_A+\eta_A e_3+\xib_A e_4,\qquad\quad\quad \D_4 e_A=\nab_4 e_A+\etab_A e_4+\xi_A e_4,\\
    \D_3 e_3&=-2\omb e_3+2\xib_B e_B,\qquad\qquad\qquad\, \D_3 e_4=2\omb e_4+2\eta_B e_B,\\
    \D_4 e_4&=-2\om e_4+2\xi_B e_B,\qquad\qquad\qquad\;\D_4 e_3=2\om e_3+2\etab_B e_B,\\
    \D_A e_B&=\nab_A e_B+\frac{1}{2}\chi_{AB} e_3+\frac{1}{2}\chib_{AB}e_4.
\end{split}
\end{align}
The following identities hold for the null frame \eqref{choiceLL}:
\begin{align}
\begin{split}\label{6.6}
    \nabla\log\Om&=\frac{1}{2}(\eta+\etab),\qquad\;\;\;\quad\om=-\nab_4(\log\Om), \qquad\quad\;\;\;\omb=0,\\ 
    \etab&=-\ze,\qquad\qquad\qquad\;\;\xi=\xib=0,
\end{split}
\end{align}
see Section 2.3 of \cite{Taylor}.
\subsection{Integral formulae}\label{ssecave}
\begin{df}\label{average}
Given a scalar function $f$ on $S:=S(u,\ub)$, we denote its average and its average free part by
\begin{equation*}
    \ov{f}:=\frac{1}{|S|}\int_{S}f\,d\ga,\qquad\quad\fc:=f-\overline{f}.
\end{equation*}
\end{df}
\begin{lem}\label{dint}
For any scalar function $f$, the following identities hold:
\begin{align*}
    e_4\left(\int_{S(u,\ub)} f d\ga\right)&=\int_{S(u,\ub)}\left(e_4(f) +\trch f\right) d\ga ,\\
    \Om^2e_3\left(\int_{S(u,\ub)} f d\ga\right)&=\int_{S(u,\ub)}\left(\Om^2e_3(f)+ \Om^2\trchb f \right) d\ga .
\end{align*}
Taking $f=1$, we obtain
\begin{equation*}
    e_4(r)=\frac{\ov{\trch}}{2}r,\qquad e_3(r)=\frac{\overline{\Om^2\trchb}}{2\Om^2}r,\label{e3e4r}
\end{equation*}
where $r$ is the \emph{area radius} defined by
\begin{equation*}
    r(u,\ub):=\sqrt{\frac{|S(u,\ub)|}{4\pi}}.
\end{equation*}
\end{lem}
\begin{proof}
See (121) and (122) in \cite{Taylor}.
\end{proof}
\subsection{Hodge systems}\label{ssec7.2}
\begin{df}\label{tensorfields}
For tensor fields defined on a $2$--sphere $S$, we denote by $\sfr_0:=\sfr_0(S)$ the set of pairs of scalar functions, $\sfr_1:=\sfr_1(S)$ the set of $1$--forms and $\sfr_2:=\sfr_2(S)$ the set of symmetric traceless $2$--tensors.
\end{df}
\begin{df}\label{def7.2}
Given $\xi\in\sk_1$ and $U\in\sk_2$, we define their Hodge dual
\begin{equation*}
    {^*\xi}_A := \in_{AB}\xi^B,\qquad {^*U}_{AB}:=\in_{AC} {U^C}_B.
\end{equation*}
\end{df}
\begin{df}
    Given $\xi,\eta\in\sk_1$ and $U,V\in\sk_2$, we denote
\begin{align*}
    \xi\cdot \eta&:= \de^{AB}\xi_A \eta_B,\qquad\qquad\qquad\qquad\qquad\quad\;\,\xi\wedge\eta:= \in^{AB} \xi_A \eta_B,\\
    (\xi \cdot U)_A&:= \de^{BC} \xi_B U_{AC},\qquad\qquad\qquad\qquad (U\wedge V)_{AB}:=\in^{AB}U_{AC}V_{CB},\\
    (\xi\hot \eta)_{AB}&:=\xi_A \eta_B +\xi_B \eta_A -\de_{AB}\xi\cdot \eta.
\end{align*}
\end{df}
\begin{df}
    For a given $\xi\in\sfr_1$, we define the following differential operators:
    \begin{align*}
        \sdiv \xi&:=\de^{AB} \nab_A\xi_B,\\
        \curl \xi&:=\in^{AB} \nab_A \xi_B,\\
        (\nab\hot\xi)_{AB}&:=\nab_A \xi_B+\nab_B \xi_A-\de_{AB}(\sdiv\xi).
    \end{align*}
\end{df}
\begin{df}\label{hodgeop}
    We define the following Hodge type operators, as introduced in section 2.2 in \cite{Ch-Kl}:
    \begin{itemize}
        \item $d_1$ takes $\sfr_1$ into $\sfr_0$ and is given by:
        \begin{equation*}
        d_1 \xi :=(\sdiv\xi,\curl\xi),
        \end{equation*}
        \item $d_2$ takes $\sfr_2$ into $\sfr_1$ and is given by:
        \begin{equation*}
        (d_2 U)_A := \nab^{B} U_{AB}, 
        \end{equation*}
        \item $d_1^*$ takes $\sfr_0$ into $\sfr_1$ and is given by:
        \begin{align*}
        d_1^*(f,f_*)_{A}:=-\nab_A f +{\in_A}^{B}\nab_B f_*,
        \end{align*}
        \item $d_2^*$ takes $\sfr_1$ into $\sfr_2$ and is given by:
        \begin{align*}
        d_2^* \xi := -\frac{1}{2} \nab \hot \xi.
        \end{align*}
    \end{itemize}
\end{df}
We have the following identities:
\begin{align}
    \begin{split}\label{dddd}
        d_1^*d_1&=-\De_1+\mathbf{K},\qquad\qquad  d_1 d_1^*=-\De_{{0}},\\
        d_2^*d_2&=-\frac{1}{2}\De_2+\mathbf{K},\qquad\quad\; d_2 d_2^*=-\frac{1}{2}(\De_1+\mathbf{K}).
    \end{split}
\end{align}
where $\mathbf{K}$ denotes the Gauss curvature on $S$. See for example (2.2.2) in \cite{Ch-Kl} for a proof of \eqref{dddd}. 
\begin{df}\label{dfdkb}
Let $f\in\sk_0$, $\xi\in\sk_1$ and $U\in\sk_2$, we define the weighted angular derivatives $\dkb$:
\begin{align*}
    \dkb U:=rd_2 U,\qquad\quad \dkb\xi:= rd_1\xi,\qquad\quad \dkb f:= rd_1^* f.
\end{align*}
We also denote for any tensor $h\in\sk_k$, $k=0,1,2$,
\begin{equation*}
    h^{(q)}:=\left(h,\,\dkb h\,, ...\, ,\,\dkb^{q} h\right),\qquad q=0,1,2,3.
\end{equation*}
\end{df}
\begin{df}\label{Lpnorms}
For a tensor field $f$ on a $2$--sphere $S$, we denote its $L^p$--norm:
\begin{equation*}
    |f|_{p,S}:= \left(\int_S |f|^p d\ga\right)^\frac{1}{p}.
\end{equation*}
We also introduce the following scale invariant $L^p$--norm:
\begin{align*}
    \|f\|_{p,S}:=|r^{-\frac{2}{p}}f|_{p,S}.
\end{align*}
\end{df}
\subsection{Null structure equations}
\begin{df}\label{renorr}
We define the following renormalized quantity:
\begin{align}
\begin{split}\label{renorslu}
\vkp&:=d_1^*(-\om,\om^\dagger)-\b,
\end{split}
\end{align}
where $\om^\dagger$ is defined by the solution of
\begin{align*}
    \nab_3\om^\dagger=\si,\qquad\quad \om^\dagger\big|_{\Si_0\setminus K}=0.
\end{align*}
We also define the following renormalized curvature components:
\begin{equation}\label{renorq}
    \rhoc:=\rho-\frac{1}{2}\hch\c\hchb,\qquad\quad \sic:=\si-\frac{1}{2}\hch\wedge\hchb.
\end{equation}
The mass aspect functions are defined as follows:
\begin{equation}\label{murenor}
    \mu:=-\sdiv\eta-\rhoc,\qquad\mub:=-\sdiv\etab-\rhoc.
\end{equation}
We also define the following modified mass aspect functions:
\begin{equation}\label{mumrenor}
    \mum:=\mu+\frac{1}{4}\trch\trchb,\qquad \mumb:=\mub+\frac{1}{4}\trch\trchb.
\end{equation}
\end{df}
\begin{prop}\label{nulles}
We have the following null structure equations:
\begin{align}
\begin{split}
\nab_4\eta&=-\chi\cdot(\eta-\etab)-\b,\\
\nab_3\etab&=-\chib\cdot(\etab-\eta)+\bb,\\
\nab_4\hch+\trch\,\hch&=-2\om\hch-\a,\\
\nab_4\trch+\frac{1}{2}(\trch)^2&=-|\hch|^2-2\om\trch,\\
\nab_3\hchb+\trchb\,\hchb&=-\aa,\\
\nab_3\trchb+\frac{1}{2}(\trchb)^2&=-|\hchb|^2,\\
\nab_4\hchb+\frac{1}{2}\trch\,\hchb&=\nab\hot\etab+2\om\hchb-\frac{1}{2}\tr\chib\,\hch+\etab\hot\etab,\\
\nab_3\hch+\frac{1}{2}\trchb\,\hch&=\nab\hot\eta-\frac{1}{2}\trch\,\hchb+\eta\hot\eta,\\
\nab_4\trchb+\frac{1}{2}\trch\trchb&=2\om\trch-2\mub+2|\etab|^2,\\
\nab_3\trch+\frac{1}{2}\trchb\trch&=-2\mu+2|\eta|^2,\\
\nab_3\om&=\rho+\frac{3}{2}|\eta-\etab|^2+\frac{1}{2}(\eta-\etab)\c(\eta+\etab)-\frac{1}{4}|\eta+\etab|^2. 
\end{split}
\end{align}
the Codazzi equations:
\begin{align}
\begin{split}
\sdiv\hch&=\frac{1}{2}\nab\trch-\frac{1}{2}\trch\,\etab-\b+\etab\c\hch,\\ 
\sdiv\hchb&=\frac{1}{2}\nab\trchb+\frac{1}{2}\trchb\,\etab+\bb-\etab\c\hchb,\label{codazzi}
\end{split}
\end{align}
the torsion equation:
\begin{equation}
\curl\eta=-\curl\etab=\sic,\label{torsion}
\end{equation}
and the Gauss equation:
\begin{equation}
    \mathbf{K}+\frac{1}{4}\trch\trchb=-\rhoc.\label{gauss}
\end{equation}
\end{prop}
\begin{proof}
This follows readily from (3.1)--(3.5) in \cite{kr} and \eqref{6.6}.
\end{proof}
\subsection{Bianchi equations}
We state the Bianchi equations for the curvature components $\a$, $\b$, $\rhoc$, $\sic$, $\bb$ and $\aa$, where $\rhoc$ and $\sic$ are defined in \eqref{renorq}.
\begin{prop}\label{Bianchiequations}
We have the following Bianchi equations:
\begin{align*}
\nab_3\a+\frac{1}{2}\trchb\,\a&=\nab\hot\b-3(\hch\rho+{^*\hch}\si)+(\ze+4\eta)\hot\b,\\
\nab_4\b+2\trch\,\b&=\sdiv\a-2\om\b+\ze\c\a,\\
\nab_3\b+\trchb\,\b&=\nab\rhoc+{^*\nab}\sic+2\hch\cdot\bb+3(\eta\rho+{^*\eta}\si)+\frac{1}{2}d_1^*(-\hch\c\hchb,\hch\wedge\hchb)\\
\nab_4\rhoc+\frac{3}{2}\trch\,\rhoc&=\sdiv\b+(\ze+2\etab)\c\b-\frac{1}{2}\hch\c(\nab\hot\etab+\etab\hot\etab)+\frac{1}{4}\trchb|\hch|^2,\\
\nab_4\sic+\frac{3}{2}\trch\,\sic&=-\sdiv{^*\b}-(\ze+2\etab)\wedge\b-\frac{1}{2}\hch\wedge(\nab\hot\etab+\etab\hot\etab),\\
\nab_3\rhoc+\frac{3}{2}\trchb\,\rhoc&=-\sdiv\bb+(\ze-2\eta)\c\bb-\frac{1}{2}\hchb\c(\nab\hot\eta+\eta\hot\eta)+\frac{1}{4}\trch|\hchb|^2,\\
\nab_3\sic+\frac{3}{2}\trchb\,\sic&=-\sdiv{^*\bb}+(\ze-2\eta)\wedge\bb+\frac{1}{2}\hchb\wedge(\nab\hot\eta+\eta\hot\eta),\\
\nab_4\bb+\trch\,\bb&=-\nab\rhoc+{^*\nab\sic}+2\om\bb+2\hchb\c\b-3(\etab\rho-{^*\etab}\si)+\frac{1}{2}d_1^*(\hch\c\hchb,\hch\wedge\hchb),\\
\nab_3\bb+2\trchb\,\bb&=-\sdiv\aa+(2\ze-\eta)\cdot\aa,\\
\nab_4\aa+\frac{1}{2}\trch\,\aa&=-\nab\hot\bb+4\om\aa-3(\hchb\rho-{^*\hchb}\sigma)+(\zeta-4\etab)\hot\bb.
\end{align*}
\end{prop}
\begin{proof}
This follows readily from (2.7) and (2.8) in \cite{AnLuk} and \eqref{6.6}. 
\end{proof}
\section{Main theorem}\label{sec8}
In the sequel of this paper, we always denote
\begin{equation}\label{dfcuvucuv}
   V:=V(u,\ub),\qquad C_u^V:=C_u\cap V,\qquad\Cb_\ub^V:=\Cb_\ub\cap V,\qquad S:=S(u,\ub).
\end{equation}
\subsection{Fundamental norms}\label{ssec8.1}
We proceed to define our main norms on a given bootstrap spacetime $\kk=V(u_0,\ub_*)$.
\subsubsection{Schematic notation \texorpdfstring{$\Gag$}{}, \texorpdfstring{$\Gab$}{} and \texorpdfstring{$\Gaw$}{}}
In order to provide a unified treatment of nonlinear terms in the equations of Propositions \ref{nulles} and \ref{Bianchiequations}, we introduce the schematic notations $\Gag$, $\Gab$ and $\Gaw$. Here, the subscripts $g$, $b$ and $w$ denote respectively \emph{good}, \emph{bad} and \emph{worst}.
\begin{df}\label{gammag}
We divide the Ricci coefficients into three parts:
\begin{align*}
    \Gag&:=\left\{\trch-\frac{1}{r},\,\trchb+\frac{4}{r},\,\hch,\,\om,\,\etab\right\},\qquad \Gab:=\left\{\eta,\,\ze,\,r^{-1}\log(2\Om)\right\},\qquad\Gaw:=\{\hchb\}.
\end{align*}
We then denote
\begin{align*}
    \Gag^{(1)}&:=(r\nab)^{\leq 1}\Gag\cup\{r\b,\,r\rhoc,\,r\sic,\,r\mub,\,r\vkp\},\\
    \Gab^{(1)}&:=(r\nab)^{\leq 1}\Gab\cup\{r\mu\},\qquad\qquad\qquad\qquad\Gaw^{(1)}:=(r\nab)^{\leq 1}\Gaw\cup\{r\bb\},
\end{align*}
and for $q=2,3$
\begin{align*}
    \Gag^{(q)}:=(r\nab)^{\leq 1}\Gag^{(q-1)},\qquad\Gab^{(q)}:=(r\nab)^{\leq 1}\Gab^{(q-1)},\qquad\Gaw^{(q)}:=(r\nab)^{\leq 1}\Gaw^{(q-1)}.
\end{align*}
\end{df}
\begin{rk}
    The justification of Definition \ref{gammag} has to do with the expected decay properties of the Ricci coefficients and curvature components, see Section \ref{sssec8.1.3} and Lemma \ref{decayGagGabGaa} below.
\end{rk}
\subsubsection{\texorpdfstring{$\mr$}{} norms (flux of curvature)}\label{secRnorms}
 In the remainder of this paper, we always denote $\de_0$ and $\de$ as two fixed constants satisfying
\begin{align*}
    0<\de_0\ll \de\ll 1.
\end{align*}
For a tensor field $h$ defined on $\cuv$ or $\ucuv$, we denote its $L^2$--flux by
\begin{align}
\begin{split}\label{dfflux}
\|h\|_{2,\cuv}^2&:=\int_{\cuv}|h|^2:=\int_{|u|}^\ub \int_{S(u,\ub')}|h|^2d\ga d\ub',\\
\|h\|_{2,\ucuv}^2&:=\int_{\ucuv}|h|^2:=\int_{-\ub}^u \int_{S(u',\ub)}\Om^2|h|^2d\ga du'.
\end{split}
\end{align}
We define the norms of flux of curvature components in the bootstrap region $\kk$. We denote
\begin{equation*}
    \mr:=\sum_{q=0}^2\big(\mr_{q}[\a,\b,\rhoc,\sic]+\mr_q[\bb]+\ur_q[\b,\rhoc,\sic,\bb]+\ur_q[\aa]\big).
\end{equation*}
where for $q=0,1,2$
\begin{align*}
    \mr_{q}[w]:=\sup_{\mathcal{K}}\mr_{q}[w](u,\ub),\qquad\ur_{q}[w]:= \sup_{\mathcal{K}}\ur_{q}[w](u,\ub).
\end{align*}
It remains to define for $q=0,1,2$
\begin{align*}
\mr_{q}[\a,\b,\rhoc,\sic](u,\ub)&:=|u|^\de\Vert r^{\frac{1}{2}-\de}(r\nab)^q(\a,\b,\rhoc,\sic)\Vert_{2,\cuv},\\
\ur_{q}[\b,\rhoc,\sic,\bb](u,\ub)&:=|u|^\de\Vert r^{\frac{1}{2}-\de}(r\nab)^q(\b,\rhoc,\sic,\bb)\Vert_{2,\ucuv},\\
\mr_q[\bb](u,\ub)&:=|u|^{\frac{1}{2}+2\de}\Vert r^{-2\de}(r\nab)^q\bb\Vert_{2,\cuv},\\
\ur_{q}[\aa](u,\ub)&:=|u|^{\frac{1}{2}+2\de}\Vert r^{-2\de}(r\nab)^q\aa\Vert_{2,\ucuv}.
\end{align*}
\subsubsection{\texorpdfstring{$\mo$}{} norms (\texorpdfstring{$L^p(S)$}{}--norms of geometric quantities)}\label{sssec8.1.3}
We first define the $L^4(S)$--norms of $\Gag$, $\Gab$ and $\Gaw$ in the bootstrap region $\kk$. We define
\begin{align*}
\mo_{[2]}:=\sum_{q=0}^2\sup_\kk\big(\mo_q(\Gag)+\mo_q(\Gab)+\mo_q(\Gaw)\big),
\end{align*}
with
\begin{align*}
\mo_0(\Gag)&:=\|r\Gag\|_{4,S},\qquad\qquad\;\;\mo_1(\Gag):=\|r\Gag^{(1)}\|_{4,S} ,\qquad\qquad\;\;\,\,\mo_2(\Gag):=\|r^{1-\de}|u|^\de\Gag^{(2)}\|_{4,S},\\
\mo_0(\Gab)&:=\|r^{1-\de_0}|u|^{\de_0}\Gab\|_{4,S},\quad\mo_1(\Gab):=\|r^{1-3\de_0}|u|^{3\de_0}\Gab^{(1)}\|_{4,S},\;\;\mo_2(\Gab):=\|r^{1-\de-2\de_0}|u|^{\de+2\de_0}\Gab^{(2)}\|_{4,S},\\
\mo_0(\Gaw)&:=\|r^{1-\de_0}|u|^{\de_0}\Gaw\|_{4,S},\;\;\mo_1(\Gaw):=\|r^{1-\de}|u|^{\de}\Gaw^{(1)}\|_{4,S},\quad\;\;\,\mo_2(\Gaw):=\|r^{\frac{1}{2}-\de-\de_0}|u|^{\frac{1}{2}+\de+\de_0}\Gaw^{(2)}\|_{4,S}.
\end{align*}
We then define the $H^2(S)$--norms for the renormalized quantities introduced in Definition \ref{renorr}:
\begin{align*}
\mo_3&:=\mo_2^2(\nab\trch)+\mo_2^2(\nab\trchb)+\mo_2^2(\vkp)+\mo_2^2(\mu)+\mo_2^2(\mub),
\end{align*}
with
\begin{align*}
    \mo_2^2(\Ga)&:=\sup_\kk \left\|r^{2-\de}|u|^\de(r\nab)^2\Ga\right\|_{2,S},\qquad \forall\; \Ga\in\left\{\nab\trch,\nab\trchb,\vkp,\mub\right\},\\
    \mo_2^2(\mu)&:=\sup_\kk \left\|r^{2-\de-3\de_0}|u|^{\de+3\de_0}(r\nab)^2\mu\right\|_{2,S}.
\end{align*}
We also introduce the following auxiliary norm:
\begin{align*}
\mo_a:=\sup_\kk\Big\|r^{3-\de-\de_0}|u|^{\de+\de_0}\nab\bb\Big\|_{2,S}+\sup_\kk\Big\|r^{2-\de_0}|u|^{1+\de_0}\nab\aa\Big\|_{2,S}.
\end{align*}
Moreover, we define
\begin{align*}
\mo_\ga:=\sup_{\kk}\frac{|\log(2\Om)|}{\log\big(\frac{3r}{|u|}\big)}+\sup_\kk\left|1-\frac{\ub-u}{2r}\right|.
\end{align*}
Finally, we denote
\begin{equation*}
    \mo:=\mo_{[2]}+\mo_3+\mo_a+\mo_\ga.
\end{equation*}
\subsubsection{\texorpdfstring{$\II_{(0)}$}{} norm (Initial data)}\label{initialO0}
We introduce the following norm on initial hypersurface $\Si_0$:
\begin{align*}
\II_{(0)}:=\sup_{\Si_0\setminus K}\left\{\sum_{q=0}^3r\|\dk^q(\Gag,\Gab,\Gaw)\|_{2,S_{(0)}(w)}+\sum_{q=0}^2r^2\|\dk^q(\a,\aa)\|_{2,S_{(0)}(w)}+\left|1-\frac{\ub-u}{2r}\right|\right\},
\end{align*}
where
\begin{align*}
    \dk\in\{r\nab_3,r\nab_4,r\nab\}.
\end{align*}
\subsection{Statement of the main theorem}\label{ssec8.3}
The goal of this paper is to prove the following theorem.
\begin{thm}[Main theorem]\label{maintheorem}
Consider an initial data set $(\Si_0,g,k)$ which is $(1,3)$--asymptotically flat in the sense of Definition \ref{defasym}. Assume that we have on $\Si_0\setminus K$:
\begin{equation}
    \II_{(0)}\leq\ep_0,
\end{equation}
where $\ep_0>0$ is a constant small enough. Then, $\Si_0\setminus K$ has a unique development $(\M,\bg)$ in its future domain of dependence with the following properties:
\begin{enumerate}
    \item $(\M,\bg)$ can be foliated by a double null foliation $(u,\ub)$. Moreover, the outgoing cones $C_u$ are complete for all $u\leq u_0$.
    \item The norms $\mo$ and $\mr$ defined in Section \ref{ssec8.1} satisfy
\begin{equation}\label{finalesti}
    \mo\les\ep_0,\qquad\quad\mr\les\ep_0.
\end{equation}
\end{enumerate}
\end{thm}
The proof of Theorem \ref{maintheorem} is given in Section \ref{ssec8.6}. It hinges on two basic theorems stated in Section \ref{ssec8.4}, concerning estimates for $\mo$ and $\mr$ norms. \\ \\
We choose $\ep_0$ and $\ep$ small enough such that
\begin{equation*}
    \ep\ll\de_0\ll \de\ll 1, \qquad \ep:=\ep_0^{\frac{2}{3}}.
\end{equation*}
Here, $A\ll B$ means that $CA<B$ where $C$ is the largest universal constant among all the constants involved in the proof via $\lesssim$.
\subsection{Main intermediate results}\label{ssec8.4}
\begin{thm}\label{M1}
Assume that
\begin{equation}
    \II_{(0)}\leq\ep_0,\qquad \mo\leq\ep,\qquad \RR\leq\ep.
\end{equation}
Then, we have
\begin{equation}
    \mr\les \ep_0.
\end{equation}
\end{thm}
Theorem \ref{M1} is proved in Section \ref{sec9}. The proof is based on the $r^p$--weighted method introduced by Dafermos and Rodnianski in \cite{Da-Ro}.
\begin{thm}\label{M3}
Assume that
\begin{align}
    \II_{(0)}\leq\ep_0,\qquad \mo\leq\ep,\qquad \RR\les\ep_0.
\end{align}
Then, we have
\begin{equation}
    \mo\les\ep_0.
\end{equation}
\end{thm}
Theorem \ref{M3} is proved in Section \ref{sec10}. The proof is done by integrating the transport equations along the outgoing and incoming null cones and applying elliptic estimate on $2$--spheres of the double null foliation of the spacetime $\kk$.
\subsection{Proof of the main theorem}\label{ssec8.6}
We now use Theorems \ref{M1} and \ref{M3} to prove Theorem \ref{maintheorem}.
\begin{df}\label{bootstrap}
Let $\aleph(\ub_*)$ the set of spacetimes $\kk=V(u_0,\ub_*)$ associated with a double null foliation $(u,\ub)$ in which we have the following bounds:
\begin{align}
    \mo\leq\varepsilon,\qquad\quad\mr\leq\ep.\label{B2}
\end{align}
\end{df}
\begin{df}\label{defboot}
We denote $\mathcal{U}$ the set of values $\ub_*$ such that $\aleph(\ub_*)\ne\emptyset$.
\end{df}
The assumption $\II_{(0)}\leq\ep_0$ and the local existence theorem imply that \eqref{B2} holds if $\ub_*$ is sufficiently close to $|u_0|$. So, we have $\mathcal{U}\ne\emptyset$.\\ \\
Define $\ub_*$ to be the supremum of the set $\mathcal{U}$. We want to prove $\ub_*=+\infty$. We assume by contradiction that $\ub_*$ is finite. In particular, we may assume $\ub_*\in\mathcal{U}$. We consider the region $\kk=V(u_0,\ub_*)$. Recall that we have
\begin{equation*}
    \II_{(0)}\leq\ep_0,
\end{equation*}
according to the assumption of Theorem \ref{maintheorem}. Applying Theorem \ref{M1}, we obtain
\begin{equation*}
    \mr\les\ep_0.
\end{equation*}
Then, we apply Theorem \ref{M3} to obtain
\begin{equation*}
    \mo\les\ep_0.
\end{equation*}
Applying local existence results in Theorem 1 of \cite{luk}, we can extend $\kk$ to $\widetilde{\kk}:={V}(u_0,\ub_*+\nu)$ for a $\nu$ sufficiently small. We denote $\widetilde{\mo}$ and $\widetilde{\mr}$ the norms in the extended region $\widetilde{\kk}$. We have
\begin{equation*}
    \widetilde{\mo}\les\ep_0,\qquad\quad \widetilde{\mr}\les\ep_0,
\end{equation*}
as a consequence of continuity. We deduce that $V(u_0,\ub_*+\nu)$ satisfies all the properties in Definition \ref{bootstrap}, and so $\aleph(\ub_*+\nu)\ne\emptyset$, which is a contradiction. Thus, we have $\ub_*=+\infty$, which implies property 1 of Theorem \ref{maintheorem}. Moreover, we have
\begin{equation*}
    \mo\les \ep_0,\qquad\quad \mr\les\ep_0,
\end{equation*}
in the whole exterior region, which implies property 2 of Theorem \ref{maintheorem}. This concludes the proof of Theorem \ref{maintheorem}.
\begin{rk}\label{nontrapping}
    We have from \eqref{finalesti}
    \begin{align*}
        \left\|\trch-\frac{1}{r}\right\|_{\infty,S}\les\frac{\ep_0}{r},\qquad\forall S\subseteq V(u_0,+\infty),
    \end{align*}
    which implies for $\ep_0$ small enough that $\trch>0$. Hence, we deduce that there is no trapped surface in the exterior region $D^+(\Si_0\setminus K)=V(u_0,+\infty)$.
\end{rk}
\section{Bootstrap assumptions and first consequences}\label{secboot}
In the rest of the paper, we always make the following bootstrap assumptions:
\begin{align}
    \mo\leq\ep,\quad\qquad\mr\leq\ep.\label{B1}
\end{align}
\subsection{First consequences of bootstrap assumptions}
In this section, we derive first consequences of \eqref{B1} in the region $\kk=V(u_0,\ub_*)$. In the sequel, the results of this section will be used frequently without explicitly mentioning them.
\begin{rk}
According to $\mo_\ga\leq\ep$, we have in the bootstrap region $\kk$
\begin{equation}
 \left|r-\frac{\ub-u}{2}\right|\les\ep\ub.
\end{equation}
Combining with $|u|\leq \ub$ and $u<0$, this yields
\begin{equation}
   |u|\leq \ub\simeq r \quad \mbox{ in }\;\kk.
\end{equation}
\end{rk}
\begin{lem}\label{decayGagGabGaa}
Under the assumption \eqref{B1}, we have the following bounds:
\begin{align*}
\|\Gag\|_{4,S}\les\frac{\ep}{r},\qquad\qquad\qquad\|\Gag^{(1)}\|_{4,S}&\les\frac{\ep}{r}, \qquad\qquad\quad\;\;\; \|\Gag^{(2)}\|_{4,S}\les\frac{\ep}{r^{1-\de}|u|^\de},\\
\|\Gab\|_{4,S}\les\frac{\ep}{r^{1-\de_0}|u|^{\de_0}},\qquad\;\,\|\Gab^{(1)}\|_{4,S}&\les\frac{\ep}{r^{1-3\de_0}|u|^{3\de_0}},\quad\;\;\|\Gab^{(2)}\|_{4,S}\les\frac{\ep}{r^{1-\de-2\de_0}|u|^{\de+2\de_0}},\\
\|\Gaw\|_{4,S}\les\frac{\ep}{r^{1-\de_0}|u|^{\de_0}},\;\,\qquad \|\Gaw^{(1)}\|_{4,S}&\les\frac{\ep}{r^{1-\de}|u|^{\de}},\qquad\quad\|\Gaw^{(2)}\|_{4,S}\les\frac{\ep}{r^{\frac{1}{2}-\de-\de_0}|u|^{\frac{1}{2}+\de+\de_0}},\\
\|\a\|_{4,S}\les\frac{\ep}{r^2},\qquad\qquad\quad\;\,\, \|\a^{(1)}\|_{2,S}&\les\frac{\ep}{r^{2-\de}|u|^\de},\qquad\quad \|\aa^{(1)}\|_{2,S}\les\frac{\ep}{r^{1-\de_0}|u|^{1+\de_0}}.
\end{align*}
\end{lem}
\begin{proof}
It follows directly from Definition \ref{gammag} and \eqref{B1}.
\end{proof}
\begin{df}\label{betterdecay}
For two quantities $X_{(1)}$ and $X_{(2)}$ satisfying
\begin{equation*}
    \|X_{(1)}\|_{4,S}\les \frac{\ep}{r^{a_1}|u|^{b_1}},\qquad\quad \|X_{(2)}\|_{4,S}\les\frac{\ep}{r^{a_2}|u|^{b_2}},
\end{equation*}
we say that $X_{(1)}$ decays better than $X_{(2)}$ if
\begin{equation*}
    a_1\geq a_2,\qquad\quad a_1+b_1\geq a_2+b_2.
\end{equation*}
Moreover, for two quantities $X_{(1)}$ and $X_{(2)}$, we denote
\begin{align*}
X_{(1)}\preceq X_{(2)}
\end{align*}
if $X_{(1)}^{(q)}$ decays better than $X_{(2)}^{(q)}$ for $q=0,1,2,3$.
\end{df}
\begin{lem}\label{better}
Under the assumption \eqref{B1}, the following holds:
\begin{enumerate}
\item We have
\begin{align*}
    \b\preceq(\rhoc,\sic)\preceq\bb\preceq\aa,\qquad \quad\Gag\preceq\Gab\preceq\Gaw.
\end{align*}
\item We also have
\begin{align*}
\Gag\c\Gag\preceq r^{-1}\Gag,\qquad \Gag\c\Gaw\preceq r^{-1}\Gaw,\qquad \Gaw\preceq \Gab^{(1)},\qquad r^{-1}\Gaw^{(1)}\preceq\aa^{(0)}.
\end{align*}
\end{enumerate}
\end{lem}
\begin{proof}
This follows readily from Definitions \ref{gammag} and \ref{betterdecay} and Lemma \ref{decayGagGabGaa}.
\end{proof}
\begin{rk}\label{Gawrk}
    In the sequel, we choose the following conventions:
\begin{itemize}
    \item Let $q\in\{0,1,2,3\}$, for a quantity $h$ satisfying the same or even better decay and regularity as $\Ga_{i}^{(q)}$, for $i=g,b,w$ , we write
    \begin{equation*}
        h\in\Ga_i^{(q)},\qquad i=g,b,w.
    \end{equation*}
    \item For a sum of schematic notations, we write
    \begin{align*}
        X_{(1)}+X_{(2)}=X_{(2)},
    \end{align*}
    if $X_{(1)}\preceq X_{(2)}$. For example, we write
    \begin{align*}
        \Gag+\Gab&=\Gab,\qquad\quad r^{-1}\Gab\c\Gaw^{(1)}+\Gab\c\aa=\Gab\c\aa,\qquad\quad \Gag\c\Gaw=\Gag\c\Gab^{(1)},\\
\b^{(0)}+\rhoc^{(0)}&=\rhoc^{(0)}=\Gag^{(1)},\qquad\qquad\quad\quad\, \b^{(1)}=\Gag^{(2)},
    \end{align*}
    where we recall that for any tensor $h$, $h^{(q)}$ is defined in Definition \ref{dfdkb}.
\end{itemize}
\end{rk}
\subsection{Commutation identities}
\begin{prop}\label{commcor}
For any scalar function $f$, we have:
\begin{align*}
[\nab_4,r\nab]f&=\Gag\cdot r\nab f,\\
[\Om^2\nab_3,r\nab]f&=\Om^2\Gaw\cdot r\nab f.
\end{align*}
\end{prop}
\begin{proof}
See Proposition 4.8.1 in \cite{Kl-Ni} and Corollary 2.22 in \cite{Shen22}.
\end{proof}
\begin{prop}\label{commutation}
We have the following schematic commutator identities
\begin{align*}
[\nab_4,r\nab]&=\Gag\cdot r\nab+\Gag^{(1)},\\
[\Om^2\nab_3,r\nab]&=\Om^2\Gaw\cdot r\nab+\Om^2\Gaw^{(1)}.
\end{align*}
\end{prop}
\begin{proof}
    See Lemma 7.3.3 in \cite{Ch-Kl} and Proposition 2.24 in \cite{Shen22}.
\end{proof}
\subsection{Main equations in schematic form}
\begin{prop}\label{null}
We have the following schematic null structure equations:
\begin{align}
\begin{split}\label{nullschematic}
\nab_4\trch+\frac{1}{2}(\trch)^2&=-2\om\trch+\Gag\c\Gag,\\
\nab_3\trchb+\frac{1}{2}(\trchb)^2&=\Gaw\c\Gaw,\\ 
\nab_4\trchb+\frac{1}{2}\trch\trchb&=2\om\trchb-2\mub+\Gag\c\Gag,\\
\nab_3\trch+\frac{1}{2}\trchb\trch&=-2\mu+\Gab\c\Gab,\\
\sdiv\hch&=\frac{1}{2}\nab\trch-\frac{1}{2}r^{-1}\etab-\b+\Gag\c\Gag,\\
\sdiv\hchb&=\frac{1}{2}\nab\trchb-\frac{1}{2}r^{-1}\etab+\bb+\Gag\c\Gaw,\\
d_1\eta&=(-\mu-\rhoc,\sic),\\
d_1\etab&=(-\mub-\rhoc,-\sic),\\
d_1^*(-\om,\om^\dagger)&=\vkp+\b,\\
\nab_3\om&=\rhoc+\Gab\c\Gab+\Gag\c\Gaw.
\end{split}
\end{align}
We also have the following transport equations for the quantities defined in Definition \ref{renorr}:
\begin{align}
\begin{split}\label{renorschematic}
\nab_3\vkp+\frac{1}{2}\trchb\,\vkp&=-r^{-1}\b+r^{-1}(\Gab\c\Gab)^{(1)},\\
\nab_4\mumc+\trch\,\mumc&=r^{-2}\,\widecheck{\trch}+r^{-1}\widecheck{\rhoc}+r^{-1}\Gab\c\Gab+r^{-1}(\Gag\c\Gab)^{(1)},\\
\nab_3\mumbc+\trchb\,\mumbc&=r^{-2}\,\widecheck{\trchb}-r^{-1}\widecheck{\rhoc}+r^{-1}(\Gaw\c\Gaw)^{(1)}.
\end{split}
\end{align}
\end{prop}
\begin{proof}
Notice that \eqref{nullschematic} follows directly from Proposition \ref{nulles} and Definition \ref{gammag}. See (6.19) in \cite{kr} and Lemma 4.3.2 in \cite{Kl-Ni} for explicit formulae of \eqref{renorschematic}. 
\end{proof}
\begin{prop}\label{bianchi}
We have the following schematic Bianchi equations:
\begin{align*}
\nab_3\a+\frac{1}{2}\trchb\,\a&=-2d_2^*\b+r^{-1}(\Gag\c\Gab)^{(1)},\\
\nab_4\b+2\trch\,\b&=d_2\a+r^{-1}(\Gag\c\Gag)^{(1)}+\Gab\c\a,\\
\nab_3\b+\trchb\,\b&=-d_1^*(\rhoc,-\sic)+r^{-1}(\Gab\c\Gaw)^{(1)},\\
\nab_4(\rhoc,-\sic)+\frac{3}{2}\trch(\rhoc,-\sic)&=d_1\b+r^{-1}(\Gag\c\Gag)^{(1)},\\
\nab_3(\rhoc,\sic)+\frac{3}{2}\trchb(\rhoc,\sic)&=-d_1\bb+r^{-1}(\Gab\c\Gaw)^{(1)},\\
\nab_4\bb+\trch\,\bb&=d_1^*(\rhoc,\sic)+r^{-1}(\Gag\c\Gaw)^{(1)},\\
\nab_3\bb+2\trchb\,\bb&=-d_2\aa+\Gab\c\aa,\\
\nab_4\aa+\frac{1}{2}\trch\,\aa&=2d_2^*\bb+\Gag\c\aa+r^{-1}(\Gab\c\Gaw)^{(1)}.
\end{align*}
\end{prop}
\begin{proof}
It follows directly from Proposition \ref{Bianchiequations}, Definition \ref{gammag} and Remark \ref{Gawrk}.
\end{proof}
\subsection{Sobolev inequalities and \texorpdfstring{$L^p$}{}--estimates}
\begin{prop}\label{sobolev}
Let $F$ be a tensor field on $S$. Then, we have
\begin{align*}
\|r^\frac{3}{2}F\|_{4,S}&\les\|r^\frac{3}{2}F\|_{4,S(u,|u|)}+\|F\|_{2,\cuv}+\|r\nab F\|_{2,\cuv}+\|r\nab_4F\|_{2,\cuv},\\
\|r|u|^\frac{1}{2}F\|_{4,S}&\les\|r^\frac{3}{2}F\|_{4,S(-\ub,\ub)}+\|\Om^{-1}F\|_{2,\ucuv}+\|\Om^{-1}r\nab F\|_{2,\ucuv}+|u|\|\Om\nab_3F\|_{2,\ucuv},
\end{align*}
where we recall that the $L^2$--flux on $\cuv$ and $\ucuv$ has been defined in \eqref{dfflux}.
\end{prop}
\begin{proof}
See Corollary 3.2.1.1 in \cite{Ch-Kl} and Corollary 4.1.1 in \cite{Kl-Ni}.
\end{proof}
\begin{prop}\label{standardsobolev}
Let $F$ be a tensor field on $S$. Then, we have
\begin{align*}
\|F\|_{\infty,S}\les\|F\|_{4,S}^\frac{1}{2}\|r\nab F\|_{4,S}^\frac{1}{2}+\|F\|_{4,S}.
\end{align*}
\end{prop}
\begin{proof}
    This is the standard Gagliardo-Nirenberg inequality.
\end{proof}
\begin{prop}\label{elliptic2d}
The following holds for all $p\in(1,+\infty)$:
\begin{enumerate}
\item Let $\phi\in\sfr_0$ be a solution of $\De\phi=f$. Then we have
\begin{align*}
    \|\nab^2 \phi\|_{p,S}+r^{-1}\|\nab\phi\|_{p,S}+r^{-2} \|\widecheck{\phi}\|_{p,S}\les \|f\|_{p,S}.
\end{align*}
\item Let $\xi\in\sfr_1$ be a solution of $d_1\xi=(f,f_*)$. Then we have
\begin{align*}
    \|\nab\xi\|_{p,S}+r^{-1}\|\xi\|_{p,S}\les\|f\|_{p,S}+\|f_*\|_{p,S}.
\end{align*}
\item Let $U\in\sfr_2$ be a solution of $d_2U=f$. Then we have
\begin{align*}
    \|\nab U\|_{p,S}+r^{-1}\|U\|_{p,S}\les\|f\|_{p,S}.
\end{align*}
\end{enumerate}
\end{prop}
\begin{proof}
See Corollary 2.3.1.1 in \cite{Ch-Kl}.
\end{proof}
\section{Hyperbolic estimates (proof of Theorem \ref{M1})}\label{sec9}
In this section, we prove Theorem \ref{M1} by the $r^p$--weighted estimate method introduced in \cite{Da-Ro} and applied to Bianchi equations in \cite{holzegel,Shen22,Shen23}.
\subsection{Schematic notation \texorpdfstring{$\F_q$}{}, \texorpdfstring{$\Fb_q$}{} and \texorpdfstring{$\O_\ell$}{}}
We introduce the following schematic notations as in \cite{Shen23}.
\begin{df}\label{wonderfuldf}
    For a quantity $X$, we denote
    \begin{equation*}
        X\in \F_q,
    \end{equation*}
    if $X$ satisfies the following estimates:
    \begin{equation*}
        \|r^\frac{q}{2}X^{(2)}\|_{2,\cuv}\leq\frac{\ep}{|u|^{\frac{1-q}{2}}},\qquad \|X^{(1)}\|_{4,S}\leq\frac{\ep}{r^\frac{q+3}{2}|u|^{\frac{1-q}{2}}}.
    \end{equation*}
    Similarly, we denote
    \begin{equation*}
        X\in\Fb_q,
    \end{equation*}
    if $X$ satisfies the following estimates:
    \begin{equation*}
        \|r^\frac{q}{2}X^{(2)}\|_{2,\ucuv}\leq\frac{\ep}{|u|^{\frac{1-q}{2}}},\qquad\|\Om X^{(1)}\|_{4,S}\leq\frac{\ep}{r^\frac{q+2}{2}|u|^{\frac{2-q}{2}}}.
    \end{equation*}
    Moreover, we denote
    \begin{align*}
        X\in \O_{\ell},
    \end{align*}
    if $X$ satisfies the following estimate:
    \begin{align*}
        \|X^{(2)}\|_{4,S}\leq\frac{\ep}{r^{\frac{\ell+1}{2}}|u|^\frac{1-\ell}{2}}.
    \end{align*}
    For $\mathbb{X}\in\{\F,\Fb,\O\}$, we also denote $X\in\mathbb{X}_q^{(i)}$, for $i=0,1,2$, if $X$ satisfies the same or even better decay and regularity properties as $\mathbb{X}_q^{(i)}$.
\end{df}
\begin{lem}\label{wonderfulrk}
Under the assumption \eqref{B1}, the following hold:
\begin{enumerate}
    \item We have
    \begin{align}
    \begin{split}\label{GagGabGawomb}
    \b,\,\rhoc,\,\sic&\in\F_{1-2\de},\qquad\qquad\qquad\qquad\qquad\;\;\;\;\;\,\bb\in\Fb_{1-2\de},\\
    \Gag&\in\O_{1-2\de},\qquad\Gab\in\O_{1-2\de-4\de_0}, \qquad\Gaw\in\O_{-2\de-2\de_0}\cap\O_{1-2\de}^{(1)}.
    \end{split}
    \end{align}
    \item We have\footnote{We have neither $\a\in\F_{1-2\de}$ nor $\aa\in\Fb_{-4\de}$ since \eqref{B1} does not control their $W^{1,4}(S)$ norms.}
    \begin{align}\label{aaa}
     \a\in\F_{1-2\de}^{(1)},\qquad\a^{(2)}\in\F_{1-2\de}^{(2)},\qquad\aa\in\Fb_{-2\de_0}^{(1)},\qquad\aa^{(2)}\in\Fb_{-4\de}^{(2)}.
    \end{align}
    \item We have
    \begin{align}\label{1GagGabGaw}
    r^{-1}\Gag^{(1)}\in\F_{1-2\de-\de_0},\qquad r^{-1}\Gab^{(1)}\in\F_{1-2\de-7\de_0},\qquad r^{-1}\Gaw^{(1)}\in\F_{-4\de}\cap\Fb_{1-2\de}.
    \end{align}
\end{enumerate}
\end{lem}
\begin{proof}
   Note that \eqref{GagGabGawomb} follows from $\mr\leq\ep$ and $\mo_{[2]}\leq\ep$. \eqref{aaa} follows directly from $\mr\leq\ep$ and $\mo_a\leq\ep$. Finally, \eqref{1GagGabGaw} follows from $\mr\leq\ep$ and $\mo_3\leq\ep$.
\end{proof}
The following theorem provides a unified treatment of all the nonlinear error terms in curvature estimates.
\begin{thm}\label{wonderfulrp}
Let $p_1,p_2,\ell\leq 1$, $p<1$ and $q\in\mathbb{N}$. We define
\begin{equation}
\begin{split}\label{deferr}
\ee_p^q\left[\psi,h\right]&:=\int_V r^p \left|\psi^{(q)}\c h^{(q)}\right|,\\
    \ee_p^q\left[\psi_{(1)},\Ga\c\psi_{(2)}\right]&:=\int_V r^p \left|\psi_{(1)}^{(q)}\c(\Ga\c\psi_{(2)})^{(q)}\right|.
\end{split}
\end{equation}
Then, we have the following properties:
\begin{enumerate}
    \item In the case $2p<p_1+p_2+\ell+1$, we have
\begin{equation}\label{wonderfulll}
    \ee^1_p\left[\F_{p_1}^{(1)},\O_\ell^{(1)}\c\F_{p_2}^{(1)}\right]\les\frac{\ep^3}{|u|^{1-p}}.
\end{equation}
    \item In the case $2p<p_1+p_2+\ell-1$, we have
\begin{equation}\label{wonderfulrr}
\ee^1_p\left[\Fb_{p_1}^{(1)},\O_\ell^{(1)}\c\Fb_{p_2}^{(1)}\right]\les\frac{\ep^3}{|u|^{1-p}}.
\end{equation}
\end{enumerate}
\end{thm}
\begin{proof}
Note that for $\mathbb{X}\in\{\F,\Fb,\O\}$, $\mathbb{X}_q^{(1)}$ has exactly the same decay and regularity properties as $\mathbb{X}_q$ in \cite{Shen23}. The proof then follows directly from Theorem 5.8 and Lemma A.1 in \cite{Shen23}.
\end{proof}
\subsection{Estimates for general Bianchi pairs}
The following lemma provides the general structure of Bianchi pairs.
\begin{lem}\label{keypoint}
Let $k=1,2$ and $a_{(1)}$, $a_{(2)}$ real numbers. Then, we have the following properties.
\begin{enumerate}
    \item If $\psi_{(1)},h_{(1)}\in\sk_k$ and $\psi_{(2)},h_{(2)}\in \sk_{k-1}$ satisfying
    \begin{align}
    \begin{split}\label{bianchi1}
    \nab_3(\psi_{(1)})+a_{(1)}\trchb\,\psi_{(1)}&=-kd_k^*(\psi_{(2)})+h_{(1)},\\
    \nab_4(\psi_{(2)})+a_{(2)}\trch\,\psi_{(2)}&=d_k(\psi_{(1)})+h_{(2)}.
    \end{split}
    \end{align}
Then, the pair $(\psi_{(1)},\psi_{(2)})$ satisfies for any real number $p$
\begin{align}
\begin{split}\label{div}
       &\bdiv(r^p |\psi_{(1)}|^2e_3)+k\bdiv(r^p|\psi_{(2)}|^2e_4)\\
       +&\left(2a_{(1)}-1-\frac{p}{2}\right)r^{p}\trchb|\psi_{(1)}|^2+k\left(2a_{(2)}-1-\frac{p}{2}\right)r^{p}\trch|\psi_{(2)}|^2\\
       =&2kr^p\sdiv(\psi_{(1)}\cdot\psi_{(2)})
       +2r^p\psi_{(1)}\cdot h_{(1)}+2kr^p\psi_{(2)}\cdot h_{(2)}+r^p\Gab|\psi_{(1)}|^2+r^p\Gag|\psi_{(2)}|^2.
\end{split}
\end{align}
    \item If $\psi_{(1)},h_{(1)}\in\sk_{k-1}$ and $\psi_{(2)},h_{(2)}\in\sk_k$ satisfying
\begin{align}
\begin{split}\label{bianchi2}
\nab_3(\psi_{(1)})+a_{(1)}\trchb\,\psi_{(1)}&=d_k(\psi_{(2)})+h_{(1)},\\
\nab_4(\psi_{(2)})+a_{(2)}\trch\,\psi_{(2)}&=-kd_k^*(\psi_{(1)})+h_{(2)}.
\end{split}
\end{align}
Then, the pair $(\psi_{(1)},\psi_{(2)})$ satisfies for any real number $p$
\begin{align}
\begin{split}\label{div2}
    &k\bdiv(r^p |\psi_{(1)}|^2e_3)+\bdiv(r^p|\psi_{(2)}|^2e_4)\\
    +&k\left(2a_{(1)}-1-\frac{p}{2}\right)r^{p}\trchb|\psi_{(1)}|^2+\left(2a_{(2)}-1-\frac{p}{2}\right)r^{p}\trch|\psi_{(2)}|^2\\
    =&2r^p\sdiv(\psi_{(1)}\cdot\psi_{(2)})
    +2kr^p\psi_{(1)}\cdot h_{(1)}+2r^p\psi_{(2)}\cdot h_{(2)}+r^p\Gab|\psi_{(1)}|^2+r^p\Gag|\psi_{(2)}|^2.
\end{split}
\end{align}
\end{enumerate}
\end{lem}
\begin{proof}
See Lemma 4.2 in \cite{Shen22}.
\end{proof}
\begin{rk}
Note that the Bianchi equations can be written as systems of equations of the type \eqref{bianchi1} and \eqref{bianchi2}. In particular
    \begin{itemize}
    \item the Bianchi pair $(\a,\b)$ satisfies \eqref{bianchi1} with $k=2$, $a_{(1)}=\frac{1}{2}$, $a_{(2)}=2$,
    \item the Bianchi pair $(\b,(\rhoc,-\sic))$ satisfies \eqref{bianchi1} with $k=1$, $a_{(1)}=1$, $a_{(2)}=\frac{3}{2}$,
    \item the Bianchi pair $((\rhoc,\sic),\bb)$ satisfies \eqref{bianchi2} with $k=1$, $a_{(1)}=\frac{3}{2}$, $a_{(2)}=1$,
    \item the Bianchi pair $(-\bb,\aa)$ satisfies \eqref{bianchi2} with $k=2$, $a_{(1)}=2$, $a_{(2)}=\frac{1}{2}$.
    \end{itemize}
\end{rk}
The following lemma allows us to obtain $|u|$--decay of curvature flux along $\Si_0\cap V(u,\ub)$.
\begin{lem}\label{gainu}
For any $R\in \{\a,\b,\rhoc,\sic,\bb,\aa\}$, we have the following estimate for $p<1$:
\begin{align*}
    \int_{\Si_0\cap V(u,\ub)}r^p|R^{(2)}|^2\les\frac{\ep_0^2}{|u|^{1-p}}.
\end{align*}
\end{lem}
\begin{proof}
We have from $\II_{(0)}\leq\ep_0$
\begin{align*}
    |R^{(2)}|_{2,S}\les\frac{\ep_0}{r}.
\end{align*}
Hence, we obtain for $p<1$
\begin{align*}
\int_{\Si_0\cap V(u,\ub)}r^p|R^{(2)}|^2\les\int_{|u|}^\infty r^p|R^{(2)}|^2_{2,S}dr\les\int_{|u|}^\infty \frac{\ep_0^2}{r^{2-p}} dr\les \frac{\ep_0^2}{|u|^{1-p}}.
\end{align*}
This concludes the proof of Lemma \ref{gainu}.
\end{proof}
\begin{figure}
  \centering
  \includegraphics[width=0.95\textwidth]{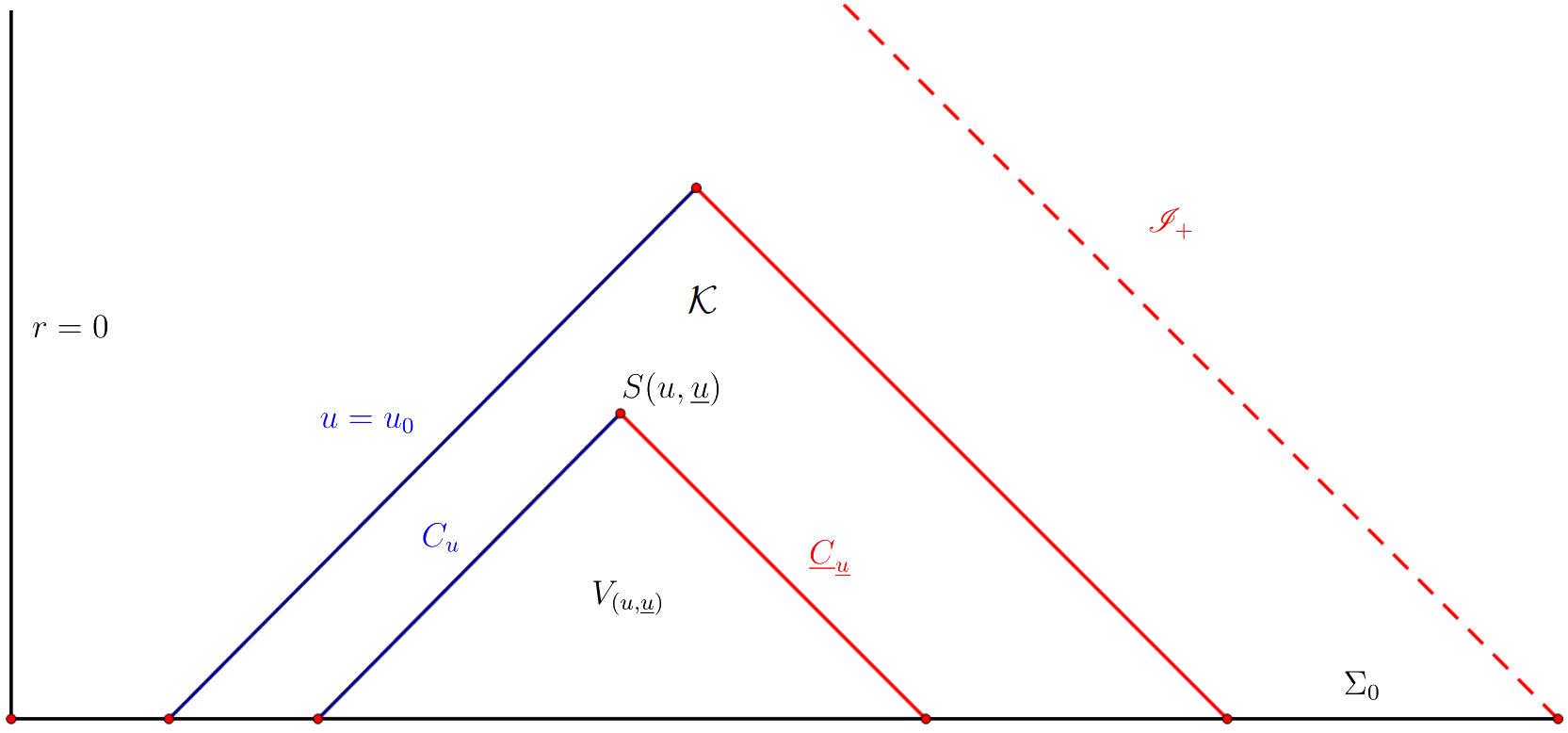}
  \caption{Domain of integration $V(u,\ub)$}\label{fig4}
\end{figure}
\begin{prop}\label{keyintegral}
Let $p<1$ and let $(\psi_{(1)},\psi_{(2)})$ be a general Bianchi pair satisfying \eqref{bianchi1} or \eqref{bianchi2} and $\psi_{(1)}\in\F_p^{(2)}$. Then, we have the following properties:
\begin{itemize}
\item In the case $2+p-4a_{(1)}>0$ and $4a_{(2)}-2-p>0$, we have
\begin{align}
\begin{split}\label{caseone}
&\int_\cuv r^p |\psi_{(1)}|^2+\int_\ucuv r^p|\psi_{(2)}|^2+\int_Vr^{p-1}|\psi_{(1)}|^2+r^{p-1}|\psi_{(2)}|^2\\ 
\les &\frac{\ep_0^2}{|u|^{1-p}}+\ee_p^0\left[\psi_{(1)},h_{(1)}\right]+\ee_p^0\left[\psi_{(2)},h_{(2)}\right].
\end{split}
\end{align}
\item In the case $2+p-4a_{(1)}\leq 0$ and $4a_{(2)}-2-p>0$, we have
\begin{align}
\begin{split}\label{casethree}
&\int_\cuv r^p |\psi_{(1)}|^2+\int_\ucuv r^p|\psi_{(2)}|^2+\int_V r^{p-1}|\psi_{(2)}|^2\\ 
\les&\frac{\ep_0^2}{|u|^{1-p}}+\int_V r^{p-1}|\psi_{(1)}|^2+\ee_p^0\left[\psi_{(1)},h_{(1)}\right]+\ee_p^0\left[\psi_{(2)},h_{(2)}\right].
\end{split}
\end{align}
\end{itemize}
\end{prop}
\begin{proof}
We have from Stokes' theorem
\begin{align*}
\int_V \bdiv\left(r^p|\psi_{(1)}|^2e_3+r^p|\psi_{(2)}|^2e_4\right)&=\int_{|u|}^\ub r^p|\psi_{(1)}|^2 \g\big(e_3,-e_4\big)d\ub'+\int_{-\ub}^ur^p|\psi_{(2)}|^2\g\big(e_4,-\Om^2 e_3\big)du'\\
&+\int_{\Si_0\cap V}r^p|\psi_{(1)}|^2\g(e_3,T)+r^p|\psi_{(2)}|^2\g(e_4,T)\\
&=2\int_\cuv r^p|\psi_{(1)}|^2+2\int_\ucuv r^p|\psi_{(2)}|^2-\int_{\Si_0\cap V}r^p\left(|\psi_{(1)}|^2+|\psi_{(2)}|^2\right).
\end{align*}
Thus, integrating \eqref{div} or \eqref{div2} in $V=V(u,\ub)$, we obtain
\begin{align*}
&\int_\cuv r^p|\psi_{(1)}|^2+\int_\ucuv r^p|\psi_{(2)}|^2-\left(2a_{(1)}-1-\frac{p}{2}\right)\int_V r^{p-1}|\psi_{(1)}|^2+\left(2a_{(2)}-1-\frac{p}{2}\right)\int_V r^{p-1}|\psi_{(2)}|^2\\
\les&\int_{\Si_0\cap V}r^p\left(|\psi_{(1)}|^2+|\psi_{(2)}|^2\right)+\ee_p^0\left[\psi_{(1)},h_{(1)}+\Gab\c\psi_{(1)}\right]+\ee_p^0\left[\psi_{(2)},h_{(2)}+\Gag\c\psi_{(2)}\right].
\end{align*}
Applying Lemma \ref{gainu}, we deduce
\begin{align*}
\int_{\Si_0\cap V}r^p\left(|\psi_{(1)}|^2+|\psi_{(2)}|^2\right)\les\frac{\ep_0^2}{|u|^{1-p}}.
\end{align*}
Recalling from Lemma \ref{wonderfulrk} that $\Gab\in\O_{1-2\de-4\de_0}$, we obtain from \eqref{wonderfulll} in Theorem \ref{wonderfulrp}
\begin{align*}
    \ee_p^0[\psi_{(1)},\Gab\c\psi_{(1)}]=\ee_p^0[\F_p^{(2)},\O_{1-2\de-4\de_0}\c\F_p^{(2)}]\les\ee_p^2[\F_p,\O_{1-2\de-4\de_0}\c\F_p]\les\frac{\ep^3}{|u|^{1-p}}.
\end{align*}
Recalling from Lemma \ref{decayGagGabGaa} that $r|\Gag|\les\ep$, we deduce
\begin{align*}
    \ee_p^0[\psi_{(2)},\Gag\c\psi_{(2)}]\les\ep\int_Vr^{p-1}|\psi_{(2)}|^2.
\end{align*}
Combining the above estimates, we infer
\begin{align*}
&\int_\cuv r^p|\psi_{(1)}|^2+\int_\ucuv r^p|\psi_{(2)}|^2-\left(2a_{(1)}-1-\frac{p}{2}\right)\int_V r^{p-1}|\psi_{(1)}|^2+\left(2a_{(2)}-1-\frac{p}{2}\right)\int_V r^{p-1}|\psi_{(2)}|^2\\
\les&\frac{\ep_0^2}{|u|^{1-p}}+\ee_p^0\left[\psi_{(1)},h_{(1)}\right]+\ee_p^0\left[\psi_{(2)},h_{(2)}\right]+\ep\int_Vr^{p-1}|\psi_{(2)}|^2.
\end{align*}
Recalling that $2a_{(2)}-1-\frac{p}{2}>0$, we deduce that for $\ep$ small enough, \eqref{caseone} and \eqref{casethree} hold in the corresponding cases. This concludes the proof of Proposition \ref{keyintegral}.
\end{proof}
\subsection{Estimates for the Bianchi pair \texorpdfstring{$(\a,\b)$}{}}\label{ssec9.1}
\begin{prop}\label{estab}
We have the following estimate:
\begin{equation}\label{abs}
\int_\cuv r^{1-2\de}|\a^{(2)}|^2+\int_\ucuv r^{1-2\de}|\b^{(2)}|^2+\int_V r^{-2\de}\left(|\a^{(2)}|^2+|\b^{(2)}|^2\right)\les\frac{\ep_0^2}{|u|^{2\de}}.
\end{equation}
\end{prop}
\begin{proof}
We recall from Proposition \ref{bianchi} that
\begin{align}
\begin{split}\label{Bianchiequationab}
\nab_3\a+\frac{1}{2}\trchb\a&=-2d_2^*\b+r^{-1}(\Gag\c\Gab)^{(1)},\\
\nab_4\b+2\trch\b&=d_2\a+r^{-1}\Gag\c\Gag^{(1)}+\Gab\c\a.
\end{split}
\end{align}
Differentiating \eqref{Bianchiequationab} by $\dkb^q$ for $q=0,1,2$ and applying Proposition \ref{commutation}, we deduce
\begin{align*}
\nab_3(\dkb^q\a)+\frac{1}{2}\trchb(\dkb^q\a)&=-\dkb_{(q)}^*(\dkb^q\b)+r^{-1}\b^{(q-1)}+r^{-1}(\Gag\c\Gab)^{(3)}+(\Gaw\c\a)^{(2)},\\
\nab_4(\dkb^q\b)+2\trch(\dkb^q\b)&=\dkb(\dkb^q\a)+r^{-1}(\Gag\c\Gag)^{(3)}+(\Gab\c\a)^{(2)},
\end{align*}
where $\dkb_{(q)}^*$ is defined by
\begin{align*}
    \dkb_{(0)}^*:=2d_2^*,\qquad \dkb_{(1)}^*:=d_1^*,\qquad \dkb_{(2)}^*:=d_1,
\end{align*}
and the term $r^{-1}\b^{(q-1)}$ comes from \eqref{dddd} and the convention $\b^{(-1)}=0$.
Note that we have from Lemma \ref{wonderfulrk}
\begin{align*}
    (\Gaw\c\a)^{(2)}&=(\Gaw^{(1)}\c\a)^{(1)}+\Gaw\c\a^{(2)}=(\O_{-2\de-2\de_0}^{(1)}\c\F_{1-2\de}^{(1)})^{(1)},\\
    (\Gab\c\a)^{(2)}&=(\Gab^{(1)}\c\a)^{(1)}+\Gab\c\a^{(2)}=(\O_{1-2\de-4\de_0}^{(1)}\c\F_{1-2\de}^{(1)})^{(1)},\\
    r^{-1}(\Gag\c\Gab)^{(3)}&=(\O_{1-2\de}\c\F_{1-2\de-7\de_0})^{(2)},\\
    r^{-1}(\Gag\c\Gag)^{(3)}&=(\O_{1-2\de}\c\F_{1-2\de-\de_0})^{(2)}.
\end{align*}
Hence, we obtain by ignoring the terms which decay better
\begin{align*}
\nab_3(\dkb^q\a)+\frac{1}{2}\trchb(\dkb^q\a)&=-\dkb_{(q)}^*(\dkb^q\b)+r^{-1}\b^{(q-1)}+(\O_{-2\de-2\de_0}^{(1)}\c\F_{1-2\de}^{(1)})^{(1)},\\
\nab_4(\dkb^q\b)+2\trch(\dkb^q\b)&=\dkb(\dkb^q\a)+(\O_{1-2\de}^{(1)}\c\F_{1-2\de-7\de_0}^{(1)})^{(1)}.
\end{align*}
Applying \eqref{caseone} with $\psi_{(1)}=\dkb^q\a$, $\psi_{(2)}=\dkb^q\b$, $a_{(1)}=\frac{1}{2}$, $a_{(2)}=2$, $p=1-2\de$ and noticing that
\begin{equation}\label{pless6}
    2+p-4a_{(1)}=1-2\de>0,\qquad 4a_{(2)}-2-p=5+2\de>0,
\end{equation}
we obtain from Lemma \ref{wonderfulrk}
\begin{align}
\begin{split}\label{abestim0}
&\int_{\cuv} r^{1-2\de} |\dkb^q\a|^2+\int_\ucuv r^{1-2\de}|\dkb^q\b|^2 +\int_V r^{-2\de}\left(|\dkb^q\a|^2+|\dkb^q\b|^2\right)\\
\les&\frac{\ep_0^2}{|u|^{2\de}}+\int_V r^{-2\de}|\dkb^q\b||\dkb^{q-1}\b|+\ee_{1-2\de}^1\left[\F_{1-2\de}^{(1)},\O_{-2\de-2\de_0}^{(1)}\c\F_{1-2\de}^{(1)}\right].
\end{split}
\end{align}
The linear term on the R.H.S of \eqref{abestim0} can be easily absorbed by Cauchy-Schwarz and induction. Moreover, applying Theorem \ref{wonderfulrp} in the case $p=1-2\de$, $p_1=1-2\de$, $p_2=1-2\de$ and $\ell=-2\de-2\de_0$, and noticing that
\begin{align*}
    p_1+p_2+\ell=2-6\de-2\de_0>1-4\de=2p-1,
\end{align*}
we obtain from \eqref{wonderfulll}
\begin{align*}
\ee_{1-2\de}^1\left[\F_{1-2\de}^{(1)},\O_{-2\de-2\de_0}^{(1)}\c\F_{1-2\de}^{(1)}\right]\les\frac{\ep^3}{|u|^{2\de}}.
\end{align*}
Injecting it into \eqref{abestim0}, we obtain \eqref{abs}. This concludes the proof of Proposition \ref{estab}.
\end{proof}
\subsection{Estimates for the Bianchi pair \texorpdfstring{$(\b,(\rhoc,-\sic))$}{}}\label{ssec9.2}
\begin{prop}\label{estbr}
We have the following estimate:
\begin{align}
\begin{split}
\int_\cuv r^{1-2\de}|\b^{(2)}|^2+\int_\ucuv r^{1-2\de}|(\rhoc,\sic)^{(2)}|^2+\int_V r^{-2\de}|(\rhoc,\sic)^{(2)}|^2\les\frac{\ep_0^2}{|u|^{2\de}}.\label{brs}
\end{split}
\end{align}
\end{prop}
\begin{proof}
We recall from Proposition \ref{bianchi}
\begin{align*}
\nab_3\b+\trchb\b&=-d_1^*(\rhoc,-\sic)+r^{-1}(\Gab\c\Gaw)^{(1)},\\
\nab_4(\rhoc,-\sic)+\frac{3}{2}\trch(\rhoc,-\sic)&=d_1\b+r^{-1}(\Gag\c\Gag)^{(1)}.
\end{align*}
Differentiating it by $\dkb^2$ and applying Proposition \ref{commutation}, we deduce
\begin{align*}
\nab_3(\dkb^2\b)+\trchb(\dkb^2\b)&=-d_1^*\big(\dkb^2(\rhoc,-\sic)\big)+r^{-1}(\Gab\c\Gaw)^{(3)}\\
\nab_4\big(\dkb^2(\rhoc,-\sic)\big)+\frac{3}{2}\trch\big(\dkb^2(\rhoc,-\sic)\big)&=d_1(\dkb^2\b)+r^{-1}(\Gag\c\Gag)^{(3)}.
\end{align*}
Applying Lemma \ref{wonderfulrk}, we obtain
\begin{align*}
\nab_3(\dkb^2\b)+\trchb(\dkb^2\b)&=-d_1^*\big(\dkb^2(\rhoc,-\sic)\big)+(\O_{1-2\de-4\de_0}\c\F_{-4\de})^{(2)},\\
\nab_4\big(\dkb^2(\rhoc,-\sic)\big)+\frac{3}{2}\trch\big(\dkb^2(\rhoc,-\sic)\big)&=d_1(\dkb^2\b)+(\O_{1-2\de}\c\F_{1-2\de-3\de_0})^{(2)}.
\end{align*}
Applying \eqref{casethree} with $\psi_{(1)}=\dkb^2\b$, $\psi_{(2)}=\dkb^2(\rhoc,-\sic)$, $a_{(1)}=1$, $a_{(2)}=\frac{3}{2}$ and $p=1-2\de$, and noticing that
\begin{equation}\label{2p4}
    2+p-4a_{(1)}=-1-2\de\leq 0,\qquad 4a_{(2)}-2-p=3+2\de>0,
\end{equation}
we infer from Lemma \ref{wonderfulrk}
\begin{align*}
&\int_{\cuv} r^{1-2\de}|\dkb^2\b|^2+\int_\ucuv r^{1-2\de}|\dkb^2(\rhoc,-\sic)|^2+\int_V r^{-2\de}|\dkb^2(\rhoc,-\sic)|^2\\
\les&\frac{\ep_0^2}{|u|^{2\de}}+\int_V r^{-2\de}|\dkb^2\b|^2+\ee_{1-2\de}^2\left[\F_{1-2\de},\O_{1-2\de-4\de_0}\c\F_{-4\de}\right].
\end{align*}
Applying \eqref{wonderfulll} in Theorem \ref{wonderfulrp}, we easily deduce
\begin{align*}
\ee_{1-2\de}^2\left[\F_{1-2\de},\O_{1-2\de-4\de_0}\c\F_{-4\de}\right]\les\frac{\ep^3}{|u|^{2\de}}.
\end{align*}
Combining with \eqref{abs}, we obtain
\begin{align*}
\int_{\cuv} r^{1-2\de} |\dkb^2\b|^2+\int_\ucuv r^{1-2\de}|\dkb^2(\rhoc,-\sic)|^2+\int_V r^{-2\de}|\dkb^2(\rhoc,-\sic)|^2\les\frac{\ep_0^2}{|u|^{2\de}}.
\end{align*}
The estimates for $\dkb^{\leq 1}\b$ and $\dkb^{\leq 1}(\rhoc,-\sic)$ can be deduced similarly. This concludes the proof of Proposition \ref{estbr}.
\end{proof}
\subsection{Estimates for the Bianchi pair \texorpdfstring{$((\rhoc,\sic),\bb)$}{}}\label{ssec9.3}
\begin{prop}\label{estrb}
We have the following estimate:
\begin{equation}
\int_\cuv r^{1-2\de}|(\rhoc,\sic)^{(2)}|^2+\int_\ucuv r^{1-2\de}|\bb^{(2)}|^2+\int_V r^{-2\de}|\bb^{(2)}|^2\les\frac{\ep_0^2}{|u|^{2\de}}.\label{rbs}
\end{equation}
\end{prop}
\begin{proof}
We recall from Proposition \ref{bianchi}
\begin{align*}
\nab_3(\rhoc,\sic)+\frac{3}{2}\trchb(\rhoc,\sic)&=-d_1\bb+r^{-1}(\Gab\c\Gaw)^{(1)},\\
\nab_4\bb+\trch\bb&=d_1^*(\rhoc,\sic)+r^{-1}(\Gag\c\Gaw)^{(1)}.
\end{align*}
Differentiating it by $\dkb^2$ and applying Proposition \ref{commutation}, we deduce
\begin{align*}
\nab_3\big(\dkb^2(\rhoc,\sic)\big)+\frac{3}{2}\trchb\big(\dkb^2(\rhoc,\sic)\big)&=-d_1(\dkb^2\bb)+r^{-1}(\Gab\c\Gaw)^{(3)},\\
\nab_4(\dkb^2\bb)+\trch(\dkb^2\bb)&=d_1^*(\dkb^2(\rhoc,\sic))+r^{-1}(\Gag\c\Gaw)^{(3)}.
\end{align*}
Noticing that we have from Lemma \ref{wonderfulrk} and \eqref{codazzi}
\begin{align*}
r^{-1}(\Gag\c\Gaw)^{(3)}&=\Gag\c\bb^{(2)}+r^{-1}\Gag^{(3)}\c\Gaw+r^{-1}\Gag^{(2)}\c\Gaw^{(1)}+r^{-1}\Gag^{(1)}\c\Gaw^{(2)}\\
&=\Gag\c\bb^{(2)}+\F_{1-2\de-7\de_0}^{(2)}\c\O_{1-2\de}^{(1)}+\F_{1-2\de-7\de_0}^{(1)}\c\O_{1-2\de}^{(2)}+\O_{1-2\de}^{(1)}\c\F_{1-2\de-7\de_0}^{(2)}\\
&=\Gag\c\bb^{(2)}+(\O_{1-2\de}^{(1)}\c\F_{1-2\de-7\de_0}^{(1)})^{(1)}.
\end{align*}
Thus, we have 
\begin{align*}
\nab_3\big(\dkb^2(\rhoc,\sic)\big)+\frac{3}{2}\trchb\big(\dkb^2(\rhoc,\sic)\big)&=-d_1(\dkb^2\bb)+(\O_{-2\de-2\de_0}\c\F_{1-2\de-7\de_0})^{(2)},\\
\nab_4(\dkb^2\bb)+\trch(\dkb^2\bb)&=d_1^*\big(\dkb^2(\rhoc,\sic)\big)+\Gag\c\bb^{(2)}+(\O_{1-2\de}^{(1)}\c\F_{1-2\de-7\de_0}^{(1)})^{(1)}.
\end{align*}
Applying \eqref{casethree} with $\psi_{(1)}=\dkb^2(\rhoc,\sic)$, $\psi_{(2)}=\dkb^2\bb$, $a_{(1)}=1$, $a_{(2)}=\frac{3}{2}$ and $p=1-2\de$, we obtain
\begin{align*}
&\int_{\cuv}r^{1-2\de}|\dkb^2(\rhoc,\sic)|^2+\int_{\ucuv}r^{1-2\de}|\dkb^2\bb|^2+\int_V r^{-2\de}|\dkb^2\bb|^2\\
\les\; &\frac{\ep_0^2}{|u|^{2\de}}+\int_V\left(r^{-2\de}|\dkb^2(\rhoc,\sic)|^2+r^{1-2\de}|\Gag||\bb^{(2)}|^2\right)+\ee_{1-2\de}^2\left[\F_{1-2\de},\O_{-2\de-2\de_0}\c\F_{1-2\de-7\de_0}\right]\\
&+\ee_{1-2\de}^1\left[\F_{-4\de}^{(1)},\O_{1-2\de}^{(1)}\c\F_{1-2\de-7\de_0}^{(1)}\right]\\
\les\; &\frac{\ep_0^2}{|u|^{2\de}}+\ep\int_V r^{-2\de}|\dkb^2\bb|^2+\ee_{1-2\de}^2\left[\F_{1-2\de},\O_{-2\de-2\de_0}\c\F_{1-2\de-7\de_0}\right]+\ee_{1-2\de}^1\left[\F_{-4\de}^{(1)},\O_{1-2\de}^{(1)}\c\F_{1-2\de-7\de_0}^{(1)}\right],
\end{align*}
where we used \eqref{brs}, the fact that $r|\Gag|\les\ep$ and Proposition \ref{elliptic2d} in the last step. Thus, we deduce for $\ep$ small enough
\begin{align*}
&\int_{\cuv}r^{1-2\de}|\dkb^2(\rhoc,\sic)|^2+\int_{\ucuv}r^{1-2\de}|\dkb^2\bb|^2+\int_V r^{-2\de}|\dkb^2\bb|^2\\
\les &\frac{\ep_0^2}{|u|^{2\de}}+\ee_{1-2\de}^2\left[\F_{1-2\de},\O_{-2\de-2\de_0}\c\F_{1-2\de-7\de_0}\right]+\ee_{1-2\de}^1\left[\F_{-4\de}^{(1)},\O_{1-2\de}^{(1)}\c\F_{1-2\de-7\de_0}^{(1)}\right].
\end{align*}
Applying \eqref{wonderfulll} in Theorem \ref{wonderfulrp}, we easily deduce
\begin{align*}
\ee_{1-2\de}^2\left[\F_{1-2\de},\O_{-2\de-2\de_0}\c\F_{1-2\de-7\de_0}\right]+\ee_{1-2\de}^1\left[\F_{-4\de}^{(1)},\O_{1-2\de}^{(1)}\c\F_{1-2\de-7\de_0}^{(1)}\right]\les\frac{\ep^3}{|u|^{2\de}}.
\end{align*}
Combining the above estimates, this concludes the proof of Proposition \ref{estrb}.
\end{proof}
\subsection{Estimates for the Bianchi pair \texorpdfstring{$(\bb,\aa)$}{}}\label{ssec9.4}
\begin{prop}\label{estba}
We have the following estimate:
\begin{equation}\label{bas}
\int_\cuv r^{-4\de}|\bb^{(2)}|^2+\int_\ucuv r^{-4\de}|\aa^{(2)}|^2+\int_V r^{-1-4\de}|\aa^{(2)}|^2\les\frac{\ep_0^2}{|u|^{1+4\de}}.
\end{equation}
\end{prop}
\begin{proof}
We have from Proposition \ref{bianchi}
\begin{align*}
\nab_3\bb+2\trchb\,\bb&=-d_2\aa+\Gab\c\aa,\\
\nab_4\aa+\frac{1}{2}\trch\,\aa&=2d_2^*\bb+\Gag\c\aa+r^{-1}(\Gab\c\Gaw)^{(1)}.
\end{align*}
Differentiating it by $\dkb^2$ and applying Proposition \ref{commutation}, we infer
\begin{align*}
\nab_3(\dkb^2\bb)+2\trchb(\dkb^2\bb)&=-d_1^*(\dkb^2\aa)+r^{-1}(\Gaw\c\Gaw)^{(3)}+(\Gab\c\aa)^{(2)},\\
\nab_4(\dkb^2\aa)+\frac{1}{2}\trch(\dkb^2\aa))&=d_1\bb+r^{-1}\bb^{(1)}+(\Gag\c\aa)^{(2)}+r^{-1}(\Gab\c\Gaw)^{(3)}.
\end{align*}
Note that we have from Lemma \ref{wonderfulrk}
\begin{align*}
(\Gab\c\aa)^{(2)}&=\Gab^{(2)}\c\aa+\Gab^{(1)}\c\aa^{(1)}+\Gab\c\aa^{(2)}=\O_{1-2\de}^{(2)}\c\Fb_{-2\de_0}^{(1)}+\O_{1-2\de}^{(1)}\c\Fb_{-4\de}^{(2)},\\
(\Gag\c\aa)^{(2)}&=\Gag^{(2)}\c\aa+\Gag^{(1)}\c\aa^{(1)}+\Gag\c\aa^{(2)}=\O_{1-2\de}^{(2)}\c\Fb_{-2\de_0}^{(1)}+\O_{1-2\de}^{(1)}\c\Fb_{-2\de_0}^{(2)}+\Gag\c\aa^{(2)}.
\end{align*}
We also have from Lemma \ref{wonderfulrk}
\begin{align*}
    r^{-1}(\Gaw\c\Gaw)^{(3)}=(\O_{-2\de-2\de_0}\c\Fb_{1-2\de})^{(2)},\qquad r^{-1}(\Gab\c\Gaw)^{(3)}=(\O_{1-2\de-7\de_0}\c\Fb_{1-2\de})^{(2)}.
\end{align*}
Hence, we deduce
\begin{align*}
\nab_3(\dkb^2\bb)+2\trchb(\dkb^2\bb)&=-d_1^*(\dkb^2\aa)+(\O_{1-2\de}^{(1)}\c\Fb_{-4\de}^{(1)})^{(1)},\\
\nab_4(\dkb^2\aa)+\frac{1}{2}\trch(\dkb^2\aa)&=d_1(\dkb^2\bb)+r^{-1}\bb^{(1)}+\Gag\c\aa^{(2)}+(\O_{1-2\de}^{(1)}\c\Fb_{-2\de_0}^{(1)})^{(1)}.
\end{align*}
Applying \eqref{casethree} with $\psi_{(1)}=\dkb^2\bb$, $\psi_{(2)}=\dkb^2\aa$, $a_{(1)}=2$, $a_{(2)}=\frac{1}{2}$ and $p=-4\de$, we obtain
\begin{align*}
&\int_{\cuv}r^{-4\de}|\dkb^2\b|^2+\int_{\ucuv} r^{-4\de}|\dkb^2\aa|^2+\int_Vr^{-1-4\de}|\dkb^2\aa|^2\\
\les&\;\frac{\ep_0^2}{|u|^{1+4\de}}+\int_V r^{-1-4\de}|\dkb^2\bb|^2+\int_V r^{-1-4\de}|\dkb^2\bb||\bb^{(1)}|+\ee_{-4\de}^1\left[\Fb_{1-2\de}^{(1)},\O_{1-2\de}^{(1)}\c\Fb_{-4\de}^{(1)}\right]\\
&+\ee_{-4\de}^1\left[\Fb_{-4\de}^{(1)},\O_{1-2\de}^{(1)}\c\Fb_{-2\de_0}^{(1)}\right]+\ee_{-4\de}^0\left[\aa^{(2)},\Gag\c\aa^{(2)}\right]\\
\les&\;\frac{\ep_0^2}{|u|^{1+4\de}}+\ee_{-4\de}^1\left[\Fb_{1-2\de}^{(1)},\O_{1-2\de}^{(1)}\c\Fb_{-4\de}^{(1)}\right]+\ee_{-4\de}^1\left[\Fb_{-4\de}^{(1)},\O_{1-2\de}^{(1)}\c\Fb_{-2\de_0}^{(1)}\right]+\ee_{-4\de}^0\left[\aa^{(2)},\Gag\c\aa^{(2)}\right],
\end{align*}
where we used \eqref{rbs} at the last step. Applying \eqref{wonderfulrr} in Theorem \ref{wonderfulrp}, we easily deduce
\begin{align*}
\ee_{-4\de}^1\left[\Fb_{1-2\de}^{(1)},\O_{1-2\de}^{(1)}\c\Fb_{-4\de}^{(1)}\right]+\ee_{-4\de}^1\left[\Fb_{-4\de}^{(1)},\O_{1-2\de}^{(1)}\c\Fb_{-2\de_0}^{(1)}\right]\les\frac{\ep^3}{|u|^{1+4\de}}.
\end{align*}
Moreover, we have from $r|\Gag|\les\ep$ and Proposition \ref{elliptic2d}
\begin{align*}
\ee_{-4\de}^0\left[\aa^{(2)},\Gag\c\aa^{(2)}\right]=\int_V r^{-4\de}|\Gag||\aa^{(2)}|^2\les\ep \int_V r^{-1-4\de}|\dkb^2\aa|^2.
\end{align*}
Combining the above estimates, we obtain for $\ep$ small enough
\begin{align*}
\int_\cuv r^{-4\de}|\dkb^2\bb|^2+\int_\ucuv r^{-4\de}|\dkb^2\aa|^2+\int_Vr^{-1-4\de}|\dkb^2\aa|^2\les\frac{\ep_0^2}{|u|^{1+4\de}}.
\end{align*}
Combining with Proposition \ref{elliptic2d}, this concludes the proof of Proposition \ref{estba}.
\end{proof}
Finally, we deduce from Propositions \ref{estab}--\ref{estba} that
\begin{equation*}
    \mr \les \ep_0,
\end{equation*}
this concludes the proof of Theorem \ref{M1}.
\section{Transport and 2D--elliptic estimates (proof of Theorem \ref{M3})}\label{sec10}
The goal of this section is to prove Theorem \ref{M3}. Throughout this section, we use the following shorthand notation:
\begin{align*}
    \int_{|u|}^\ub h(u,\ub):=\int_{|u|}^\ub h(u,\ub')d\ub',\qquad\quad \int_{-\ub}^u h(u,\ub):=\int_{-\ub}^u h(u',\ub) du'. 
\end{align*}
\subsection{Preliminaries}
We recall the evolution lemma, which will be used repeatedly throughout this section.
\begin{lem}\label{evolutionlemma}
Under the assumption $\mo\leq\ep$, the following holds:
\begin{enumerate}
\item Let $U,F$ be $k$--covariant $S$--tangent tensor fields satisfying the outgoing evolution equation
\begin{equation*}
    \nab_4 U+\frac{\la}{2}\,\ov{\trch}\, U= F.
\end{equation*}
We have for $p\in[1,+\infty]$
\begin{equation*}
\|r^{\la}U\|_{p,S}\les \|r^{\la}U\|_{p,S(u,|u|)}+\int_{|u|}^{\ub} \|r^{\la}F\|_{p,S}.
\end{equation*}
\item Let $V,\uf$ be $k$--covariant $S$--tangent tensor fields satisfying the incoming evolution equation
\begin{equation*}
\Om^2\nab_3 V+\frac{\la}{2}\,\ov{\Om^2\trchb}\, V=\uf.
\end{equation*}
We have for $p\in[1,+\infty]$
\begin{equation*}
    \|r^{\la}V\|_{p,S} \les \|r^{\la}V\|_{p,S(-\ub,\ub)}+\int_{-\ub}^{u}\|r^{\la}\underline{F}\|_{p,S}.
\end{equation*}
\end{enumerate}
\end{lem}
\begin{proof}
The proof is largely analogous to Lemma 4.1.5 in \cite{Kl-Ni}. See also Propositions 10.1 and 10.2 in \cite{Taylor}.
\end{proof}
The following lemma will be useful to estimate the $L^2(S)$--norms.
\begin{lem}\label{useful}
    We have for $2\la>q+1$
    \begin{align*}
        \int_{|u|}^\ub r^{\la}\|\F_q^{(2)}\|_{2,S}\les\frac{\ep}{|u|^\frac{1-q}{2}}r^\frac{2\la-1-q}{2},
    \end{align*}
    and for $2\la<q+1$
    \begin{align*}
        \int_{|u|}^\ub r^{\la}\|\F_q^{(2)}\|_{2,S}\les\ep|u|^{\la-1}.
    \end{align*}
    We also have for $2\la>-1$
    \begin{align*}
        \int_{-\ub}^u |u|^{\la}\|\Om\Fb_q^{(2)}\|_{2,S}\les\frac{\ep}{|u|^\frac{1-q}{2}}r^\frac{2\la-1-q}{2}.
    \end{align*}
\end{lem}
\begin{proof}
    We have for $2\la>q+1$
    \begin{align*}
    \int_{|u|}^\ub r^{\la}\|\F_q^{(2)}\|_{2,S}\les\left(\int_{|u|}^\ub r^{2\la-2-q}\right)^\frac{1}{2}\left(\int_{|u|}^\ub r^q|\F_q^{(2)}|_{2,S}^2\right)^\frac{1}{2}\les \frac{\ep}{|u|^\frac{1-q}{2}}r^\frac{2\la-1-q}{2},
    \end{align*}
    and for $2\la<q+1$
     \begin{align*}
    \int_{|u|}^\ub r^{\la}\|\F_q^{(2)}\|_{2,S}\les\frac{\ep}{|u|^\frac{1-q}{2}}|u|^\frac{2\la-1-q}{2}\les\ep|u|^{\la-1}.
    \end{align*}
    We also have for $2\la>-1$
    \begin{align*}
    \int_{-\ub}^u |u|^{\la}\|\Om\Fb_q^{(2)}\|_{2,S}\les r^{-\frac{q+2}{2}}\left(\int_{-\ub}^u|u|^{2\la}\right)^\frac{1}{2}\left(\int_{-\ub}^ur^q|\Om\Fb_q^{(2)}|_{2,S}^2\right)^\frac{1}{2}\les \frac{\ep}{|u|^\frac{1-q}{2}}r^\frac{2\la-1-q}{2}.
    \end{align*}
    This concludes the proof of Lemma \ref{useful}.
\end{proof}
\subsection{Estimates for curvature}
\begin{prop}\label{R1estima}
We have the following estimates:
\begin{align*}
    \|(\b,\rhoc,\sic)^{(1)}\|_{4,S}\les\frac{\ep_0}{r^{2-\de}|u|^\de},\qquad\qquad\|\Om\bb^{(1)}\|_{4,S}\les\frac{\ep_0}{r^{\frac{3}{2}-\de}|u|^{\frac{1}{2}+\de}}.
\end{align*}
\end{prop}
\begin{proof}
    It follows directly from Theorem \ref{M1} and Proposition \ref{sobolev}.
\end{proof}
\begin{prop}\label{R0estima}
We have the following estimates:
\begin{align*}
\|(\a,\b,\rhoc,\sic)\|_{4,S}\les\frac{\ep_0}{r^{2}},\qquad \|\bb\|_{4,S}\les\frac{\ep_0}{r^{2-\de}|u|^\de}.
\end{align*}
\end{prop}
\begin{proof}
We have from Proposition \ref{bianchi}
\begin{align*}
    \nab_3(\a,\b,\rhoc,\sic)=r^{-1}(\a,\b,\rhoc,\sic,\bb)^{(1)}+r^{-1}(\Gab\c\Gaw)^{(1)},
\end{align*}
which implies
\begin{equation}
    \Om^2\nab_3(\a,\b,\rhoc,\sic)=r^{-1}\Om^2\bb^{(1)}+r^{-1}\Om^2(\Gab\c\Gaw)^{(1)}.
\end{equation}
Applying Lemma \ref{evolutionlemma} and Proposition \ref{R1estima}, we obtain
\begin{align*}
    \|(\a,\b,\rhoc,\sic)\|_{4,S}&\les\|(\a,\b,\rhoc,\sic)\|_{4,S(-\ub,\ub)}+\int_{-\ub}^u\Om^2\left(\|\bb^{(1)}\|_{4,S}+\|(\Gab\c\Gaw)^{(1)}\|_{4,S}\right)\\
    &\les\frac{\ep_0}{r^2}+\int_{-\ub}^u\frac{r^{\de_0}}{|u|^{\de_0}}\left(\frac{\ep_0}{r^{\frac{5}{2}-\de}|u|^{\frac{1}{2}+\de}}+\frac{\ep^2}{r^{3-\de-2\de_0}|u|^{\de+2\de_0}}\right)\\
    &\les\frac{\ep_0}{r^2}.
\end{align*}
Next, we have from Proposition \ref{bianchi}
\begin{align*}
    \nab_4\bb+\trch\,\bb=r^{-1}(\rhoc,\sic)^{(1)}+r^{-1}(\Gag\c\Gaw)^{(1)}.
\end{align*}
Applying Lemma \ref{evolutionlemma} and Proposition \ref{R1estima}, we infer
\begin{align*}
    \|r^2\bb\|_{4,S}&\les\|r^2\bb\|_{4,S(u,|u|)}+\int_{|u|}^\ub \|r(\rhoc,\sic)^{(1)}\|_{4,S}+\|r(\Gag\c\Gaw)^{(1)}\|_{4,S}\\
    &\les\ep_0+\int_{|u|}^\ub \frac{\ep_0}{r^{1-\de}|u|^\de}+\frac{\ep^2}{r^{1-\de}|u|^\de}\\
    &\les\frac{\ep_0}{|u|^\de}r^\de.
\end{align*}
This concludes the proof of Proposition \ref{R0estima}.
\end{proof}
\begin{prop}\label{bbaarecover}
We have the following estimates:
\begin{align*}
    \|\bb^{(1)}\|_{2,S}\les\frac{\ep_0}{r^{2-\de-\de_0}|u|^{\de+\de_0}},\qquad \|\aa^{(1)}\|_{2,S}\les\frac{\ep_0}{r^{1-\de_0}|u|^{1+\de_0}}.
\end{align*}
\end{prop}
\begin{proof}
We have from Proposition \ref{bianchi}
\begin{align*}
    \nab_4\bb+\trch\,\bb=r^{-1}(\rhoc,\sic)^{(1)}+r^{-1}(\Gag\c\Gaw)^{(1)}.
\end{align*}
Differentiating it by $r\nab$ and applying Proposition \ref{commutation}, we obtain
\begin{align*}
    \nab_4(r\nab\bb)+\trch(r\nab\bb)=r^{-1}(\rhoc,\sic)^{(2)}+r^{-1}(\Gag\c\Gaw)^{(2)}.
\end{align*}
Note that we have from Proposition \ref{estrb}
\begin{align}\label{rhocE}
    (\rhoc,\sic)\in\frac{\ep_0}{\ep}\F_{1-2\de}.
\end{align}
Applying Proposition \ref{elliptic2d} and Lemmas \ref{evolutionlemma} and \ref{useful}, we infer
\begin{align*}
    \|r^2\bb^{(1)}\|_{2,S}&\les \|r^2\bb^{(1)}\|_{2,S(u,|u|)}+\int_{|u|}^\ub\|r(\rhoc,\sic)^{(2)}\|_{2,S}+\frac{\ep^2}{r^{1-\de-\de_0}|u|^{\de+\de_0}}\\
    &\les\ep_0+\frac{\ep^2}{|u|^{\de+\de_0}}r^{\de+\de_0}+\frac{\ep_0}{\ep}\int_{|u|}^\ub r\|\F_{1-2\de}^{(2)}\|_{2,S},\\
    &\les \frac{\ep_0}{|u|^{\de+\de_0}}r^{\de+\de_0}.
\end{align*}
Next, we have from Proposition \ref{bianchi}
\begin{align*}
\nab_4\aa+\frac{1}{2}\trch\,\aa=r^{-1}\bb^{(1)}+r^{-1}(\Gab\c\Gaw)^{(1)}.
\end{align*}
Differentiating it by $r\nab$ and applying Proposition \ref{commutation}, we deduce
\begin{align*}
\nab_4(r\nab\aa)+\frac{1}{2}\trch(r\nab\aa)=r^{-1}\bb^{(2)}+(\Gag\c\aa)^{(1)}+r^{-1}(\Gab\c\Gaw)^{(2)}.
\end{align*}
Note that we have from Proposition \ref{estba}
\begin{align*}
    \bb\in\frac{\ep_0}{\ep}\F_{-4\de}.
\end{align*}
We also have from Lemma \ref{decayGagGabGaa} and Proposition \ref{standardsobolev}
\begin{align*}
    \|(\Gag\c\aa)^{(1)}\|_{2,S}\les\|\Gag^{(1)}\|_{4,S}\|\aa\|_{4,S}+\|\Gag\|_{\infty,S}\|\aa^{(1)}\|_{2,S}\les\frac{\ep^2}{r^{1-\de_0}|u|^{1+\de_0}}.
\end{align*}
Then, applying Lemmas \ref{evolutionlemma} and \ref{useful}, we infer
\begin{align*}
    \|r\aa^{(1)}\|_{2,S}&\les\|r\aa^{(1)}\|_{2,S(u,|u|)}+\int_{|u|}^\ub \|\bb^{(2)}\|_{2,S}+\frac{\ep^2}{r^{1-\de_0}|u|^{1+\de_0}}+\frac{\ep^2}{r^{\frac{3}{2}-\de-\de_0}|u|^{\frac{1}{2}+\de+\de_0}}\\
    &\les\frac{\ep_0}{|u|}+\frac{\ep^2}{|u|^{1+\de_0}}r^{\de_0}+\frac{\ep^2}{|u|}+\frac{\ep_0}{\ep}\int_{|u|}^\ub\|\F_{-4\de}^{(2)}\|_{2,S}\\
    &\les\frac{\ep_0 r^{\de_0}}{|u|^{1+\de_0}}.
\end{align*}
This concludes the proof of Proposition \ref{bbaarecover}.
\end{proof}
\subsection{Estimates for \texorpdfstring{$\vkp$}{}, \texorpdfstring{$\om$}{} and \texorpdfstring{$\log(2\Om)$}{}}\label{ssec10.4}
\begin{prop}\label{prop10.4}
We have the following estimates:
\begin{align}
\begin{split}\label{estombom}
\|\vkp^{(1)}\|_{4,S}\les\frac{\ep_0}{r^2},\qquad\quad \|\vkp^{(2)}\|_{2,S}\les\frac{\ep_0}{r^{2-\de}|u|^\de}.
\end{split}
\end{align}
\end{prop}
\begin{proof}
We have from Proposition \ref{null}
\begin{align}
\begin{split}\label{transvkp}
\nab_3{\varkappa}+\frac{1}{2}\trchb\,\varkappa=-r^{-1}\b+r^{-1}(\Gab\cdot\Gaw)^{(1)}.
\end{split}
\end{align}
Differentiating \eqref{transvkp} by $\dkb$ and applying Proposition \ref{commutation}, we obtain
\begin{align}\label{dkbvkp}
    \nab_3(\dkb\vkp)+\frac{1}{2}\trchb(\dkb\vkp)=r^{-1}\b^{(1)}+r^{-1}(\Gab\cdot\Gaw)^{(2)}.
\end{align}
Applying Lemma \ref{evolutionlemma}, we deduce
\begin{align*}
    \|r(\dkb\vkp)\|_{4,S}&\les\|r(\dkb{\vkp})\|_{4,S(-\ub,\ub)}+\int_{-\ub}^{u}\|\Om^2\b^{(1)}\|_{4,S}+\|\Om^2(\Gab\c\Gaw)^{(2)}\|_{4,S}\\
    &\les\frac{\ep_0}{r}+\int_{-\ub}^{u}\frac{\ep_0}{r^{2-\de-\de_0}|u|^{\de+\de_0}}+\frac{\ep^2}{r^{\frac{3}{2}-\de-2\de_0}|u|^{\frac{1}{2}+\de+2\de_0}}\\
    &\les\frac{\ep_0}{r}.
\end{align*}
Next, differentiating \eqref{dkbvkp} by $\dkb$ and applying Proposition \ref{commutation}, we infer
\begin{align*}
\nab_3(\dkb^2\vkp)+\frac{1}{2}\trchb(\dkb^2\vkp)=r^{-1}\b^{(2)}+r^{-1}(\Gab\cdot\Gaw)^{(3)}.
\end{align*}
Note that we have from Proposition \ref{estab} and Lemma \ref{wonderfulrk}
\begin{align}\label{Gaw3f}
    (\Om\b)^{(2)}\in\frac{\ep_0}{\ep}\Fb_{1-2\de}^{(2)},\qquad r^{-1}\Gab^{(3)},r^{-1}\Gaw^{(3)}\in\Fb_{1-2\de}^{(2)}.
\end{align}
Applying Lemmas \ref{evolutionlemma} and \ref{useful}, we have
\begin{align*}
    \|r\dkb^2\vkp\|_{2,S}&\les\|r\dkb^2\vkp\|_{2,S(-\ub,\ub)}+\int_{-\ub}^{u}\|\Om^2(\Om\b)^{(2)}\|_{2,S}+\|\Om^2(\Gab\c\Gaw)^{(3)}\|_{2,S}\\
    &\les\frac{\ep_0}{r}+\frac{\ep_0}{\ep}r^{\de_0}\int_{-\ub}^u |u|^{-\de_0}\|\Fb_{1-2\de}^{(2)}\|_{2,S}+\ep r^{\frac{\de_0+\de}{2}}\int_{-\ub}^u |u|^{-\frac{\de+\de_0}{2}}\|\Fb_{1-2\de}^{(2)}\|_{2,S},\\
    &\les\frac{\ep_0}{r}+\frac{\ep_0}{r^{1-\de}|u|^\de}+\frac{\ep^2}{r^{1-\de}|u|^\de}\\
    &\les\frac{\ep_0}{r^{1-\de}|u|^\de}.
\end{align*}
This concludes the proof of Proposition \ref{prop10.4}.
\end{proof}
\begin{prop}\label{prop10.3}
We have the following estimates:
\begin{align}
\begin{split}
    \|\om^{(1)}\|_{4,S}&\les\frac{\ep_0}{r},\qquad\quad\|\om^{(2)}\|_{4,S}\les\frac{\ep_0}{r^{1-\de}|u|^\de}.
\end{split}
\end{align}
\end{prop}
\begin{proof}
We have from Proposition \ref{null}
\begin{align}
\nab_3\om=\rhoc+\Gab\c\Gab+\Gag\c\Gaw.\label{34omega}
\end{align}
Applying Lemma \ref{evolutionlemma}, we obtain
\begin{align*}
\|\om\|_{4,S}&\les\|\om\|_{4,S(-\ub,\ub)}+\int_{-\ub}^u\|\Om^2\rhoc\|_{4,S}+\|\Om^2\Gab\c\Gab\|_{4,S}+\|\Om^2\Gag\c\Gaw\|_{4,S}\\
&\les\frac{\ep_0}{r}+\int_{-\ub}^u\frac{\ep_0}{r^{2}}+\frac{\ep^2}{r^{1-3\de_0}|u|^{3\de_0}}+\frac{\ep^2}{r^{2-\de}|u|^{\de}}\\
&\les\frac{\ep_0}{r}.
\end{align*}
Applying Propositions \ref{null}, \ref{elliptic2d} and \ref{prop10.4}, we deduce
\begin{align*}
    \|\dkb\om\|_{4,S}\les\|r(\vkp,\b)\|_{4,S}\les\frac{\ep_0}{r},\qquad\|\dkb^2\om\|_{4,S}\les\|r(\vkp,\b)^{(1)}\|_{4,S}\les\frac{\ep_0}{r^{1-\de}|u|^\de}.
\end{align*}
This concludes the proof of Proposition \ref{prop10.3}.
\end{proof}
\begin{prop}\label{propOmc}
We have the following estimates:
\begin{align}
\begin{split}\label{estOmc}
\|\log(2\Om)\|_{\infty,S} &\les\ep_0\log\left(\frac{3r}{|u|}\right),\qquad\quad \|\dkb\log(2\Om)\|_{4,S}\les\ep_0\frac{r^{\de_0}}{|u|^{\de_0}},\\
\|\dkb^2\log(2\Om)\|_{4,S}&\les\ep_0\frac{r^{\de}}{|u|^{\de}},\qquad\qquad\quad\;
\|\dkb^3\log(2\Om)\|_{2,S}\les\ep_0\frac{r^{\de+\de_0}}{|u|^{\de+\de_0}}.
\end{split}
\end{align}
\end{prop}
\begin{proof}
We recall from \eqref{6.6} that
\begin{equation}\label{logOmtrans}
    \nab_4(\log(2\Om))=-\om.
\end{equation}
Applying Lemma \ref{evolutionlemma} and Propositions \ref{prop10.3} and \ref{standardsobolev}, we infer
\begin{align*}
    \|\log(2\Om)\|_{\infty,S}\les\|\log(2\Om)\|_{\infty,S(u,|u|)}+\int_{|u|}^\ub \|\om\|_{\infty,S} d\ub\les\ep_0+\int_{|u|}^\ub\frac{\ep_0}{r}d\ub \les\ep_0\log\left(\frac{3r}{|u|}\right).
\end{align*}
Next, differentiating \eqref{logOmtrans} by $\dkb$ and applying Proposition \ref{commcor}, we obtain
\begin{align*}
    \nab_4(\dkb\log(2\Om))=\om^{(1)}+r\Gag\c\Gab.
\end{align*}
Applying Lemma \ref{evolutionlemma} and Proposition \ref{prop10.3}, we deduce
\begin{align*}
    \|\dkb\log(2\Om)\|_{4,S}\les \ep_0+\int_{|u|}^\ub\frac{\ep_0}{r}+\frac{\ep^2}{r^{1-\de_0}|u|^{\de_0}}\les\ep_0 \frac{r^{\de_0}}{|u|^{\de_0}}.
\end{align*}
Then, differentiating \eqref{logOmtrans} by $\dkb^2$ and applying Proposition \ref{commutation}, we have
\begin{align}\label{lapOm}
    \nab_4(\dkb^2\log(2\Om))=\om^{(2)}+r(\Gag\c\Gab)^{(1)}.
\end{align}
Applying Lemma \ref{evolutionlemma} and Proposition \ref{prop10.3}, we obtain
\begin{align*}
    \|\dkb^2\log(2\Om)\|_{4,S}\les\ep_0+\int_{|u|}^\ub\frac{\ep_0}{r^{1-\de}|u|^\de}+\frac{\ep^2}{r^{1-\de}|u|^\de}\les\ep_0\frac{r^\de}{|u|^\de}.
\end{align*}
Finally, differentiating \eqref{lapOm} by $\dkb$ and applying Propositions \ref{commcor} and \ref{null}, we have
\begin{align*}
    \nab_4(\dkb^3\log(2\Om))=\om^{(3)}+r(\Gag\c\Gab)^{(2)}=(\vkp,\b)^{(2)}+r(\Gag\c\Gab)^{(2)}.
\end{align*}
Notice that we have from Proposition \ref{estbr}
\begin{align*}
    \b\in\frac{\ep_0}{\ep}\F_{1-2\de}.
\end{align*}
Applying Lemma \ref{evolutionlemma} and Proposition \ref{prop10.4}, we deduce
\begin{align*}
\|\dkb^3\log(2\Om)\|_{2,S}&\les\ep_0+\int_{|u|}^\ub \|r(\vkp,\b)^{(2)}\|_{2,S}+\|r(\Gag\c\Gab)^{(2)}\|_{2,S}\\
&\les\ep_0+\frac{\ep_0}{\ep}\int_{|u|}^\ub r\|\F_{1-2\de}^{(2)}\|_{2,S}+\int_{|u|}^\ub\frac{\ep_0}{r^{1-\de}|u|^\de}+\frac{\ep^2}{r^{1-\de-\de_0}|u|^{\de+\de_0}}\\
&\les\ep_0+\frac{\ep_0}{\ep}\frac{\ep}{|u|^\de}r^\de+\frac{\ep_0}{|u|^\de}r^\de+\frac{\ep^2}{|u|^{\de+\de_0}}r^{\de+\de_0},\\
&\les\frac{\ep_0}{|u|^{\de+\de_0}}r^{\de+\de_0}.
\end{align*}
This concludes the proof of Proposition \ref{propOmc}.
\end{proof}
\subsection{Estimates for \texorpdfstring{$\widecheck{\trch}$}{} and \texorpdfstring{$\widecheck{\trchb}$}{}}\label{ssec10.5}
\begin{prop}\label{prop10.5}
We have the following estimates:
\begin{align}
    \begin{split}\label{omtrch}
    \|\trchc^{(1)}\|_{4,S}\les\frac{\ep_0}{r},\qquad\|\trchc^{(2)}\|_{4,S}\les\frac{\ep_0}{r^{1-\de}|u|^\de},\qquad\|\trchc^{(3)}\|_{2,S}\les\frac{\ep_0}{r^{1-\de}|u|^\de}.
    \end{split}
\end{align}
\end{prop}
\begin{proof}
We recall the following equation from Proposition \ref{null}:
\begin{equation}\label{4trchi}
\nab_4(\trch)+\frac{1}{2}(\trch)^2=r^{-1}\om^{(0)}+\Gag\c\Gag. 
\end{equation}
Differentiating \eqref{4trchi} by $\dkb$ and applying Proposition \ref{commcor}, we obtain
\begin{align}
\begin{split}
\nab_4(\dkb\trch)+\trch(\dkb\trch)=r^{-1}\om^{(1)}+\Gag\c\Gag^{(1)}.\label{ovchi}
\end{split}
\end{align}
Applying Lemma \ref{evolutionlemma} and Proposition \ref{prop10.4}, we obtain
\begin{align*}
\|r^2\dkb\trch\|_{4,S}&\les \|r^2\dkb\trch\|_{4,S(u,|u|)}+\int_{|u|}^{\ub}\|r\om^{(1)}\|_{4,S}+\|r^2\Gag\cdot\Gag^{(1)}\|_{4,S}\\
&\les\ep_0 |u|+\int_{|u|}^{\ub}\ep_0+\ep^2\\
&\les\ep_0r.
\end{align*}
Next, differentiating \eqref{ovchi} by $\dkb$ and applying Proposition \ref{commutation}, we infer
\begin{align}\label{trch2}
\nab_4(\dkb^2\trch)+\trch(\dkb^2\trch)=r^{-1}\om^{(2)}+(\Gag\c\Gag)^{(2)}.
\end{align}
Applying Lemma \ref{evolutionlemma} and Proposition \ref{prop10.4}, we easily obtain
\begin{align*}
\|r^2\dkb^2\trch\|_{4,S}\les \|r^2\dkb^2\trch\|_{4,S(u,|u|)}+\int_{|u|}^{\ub}\frac{\ep_0}{|u|^\de}r^\de+\frac{\ep^2}{|u|^\de}r^\de\les\frac{\ep_0}{|u|^\de}r^{1+\de}.
\end{align*}
Finally, differentiating \eqref{trch2} and applying Proposition \ref{commcor} and Proposition \ref{elliptic2d}, we deduce
\begin{align*}
\nab_4(\dkb^3\trch)+\trch(\dkb^3\trch)=(\vkp,\b)^{(2)}+(\Gag\c\Gag)^{(3)}.
\end{align*}
Notice that we have from Proposition \ref{estbr}
\begin{align*}
    \b^{(2)}\in \frac{\ep_0}{\ep}\F_{1-2\de}^{(2)},\qquad r^{-1}\Gag^{(3)}\in \F_{1-2\de}^{(2)}+\mathbb{G},\qquad \|\mathbb{G}\|_{2,S}\les\frac{\ep}{r^{2-\de}|u|^\de}.
\end{align*}
Applying Lemmas \ref{evolutionlemma} and \ref{useful}, we infer
\begin{align*}
    \|r^2\dkb^3\trch\|_{2,S}&\les\|r\dkb^3\trch\|_{2,S(u,|u|)}+\int_{|u|}^\ub r^2\|\b^{(2)}\|_{2,S}+\ep r\|\Gag^{(3)}\|_{2,S}+\frac{\ep_0}{|u|^\de}r^\de\\
    &\les\ep_0 |u|+\frac{\ep_0}{|u|^\de}r^{1+\de}+\frac{\ep_0}{\ep}\int_{|u|}^\ub r^2\|\F_{1-2\de}^{(2)}\|_{2,S}\\
    &\les\frac{\ep_0}{|u|^\de}r^{1+\de}.
\end{align*}
This concludes the proof of Proposition \ref{prop10.5}.
\end{proof}
\begin{prop}\label{prop10.6}
We have the following estimates:
\begin{align}
\begin{split}\label{esttrchbc}
\|\trchbc^{(2)}\|_{4,S}\les\frac{\ep_0}{r},\qquad\quad\|\trchbc^{(3)}\|_{2,S}\les\frac{\ep_0}{r^{1-\de}|u|^\de}.
\end{split}
\end{align}
\end{prop}
\begin{proof}
We have from Proposition \ref{null}:
\begin{equation}\label{nab3Om-1}
    \nab_3\trchb+\frac{1}{2}(\trchb)^2=\Gaw\c\Gaw.
\end{equation}
Differentiating \eqref{nab3Om-1} by $\dkb^2$ and applying Proposition \ref{commutation}, we obtain
\begin{align}
\begin{split}
\nab_3(\dkb^2\trchb)+\trchb(\dkb^2\trchb)=r^{-1}(\nab\log\Om)^{(1)}+(\Gaw\cdot\Gaw)^{(2)}.\label{ovchib}
\end{split}
\end{align}
Applying Lemmas \ref{evolutionlemma} and Proposition \ref{propOmc}, we infer
\begin{align*}
\|r^{2}\dkb^2\trchb\|_{4,S}&\les\|r^{2}\dkb^2\trchb\|_{4,S(-\ub,\ub)}+\int_{-\ub}^u\|r\Om^2(\nab\log\Om)^{(1)}\|_{4,S}+\|r^{2}\Om^2(\Gaw\cdot\Gaw)^{(2)}\|_{4,S}\\
&\les\ep_0r+\int_{-\ub}^u\frac{\ep_0}{|u|^{2\de_0}}r^{2\de_0}+r^2\frac{\ep^2}{r^{\frac{3}{2}-2\de-2\de_0}|u|^{\frac{1}{2}+2\de+2\de_0}}\\
&\les\ep_0r.
\end{align*}
Next, differentiating \eqref{ovchib} by $\dkb$ and applying Proposition \ref{commcor}, we deduce
\begin{align*}
\nab_3\big(\dkb^3\trchb\big)+\trchb(\dkb^3\trchb)=r^{-1}(\nab\log\Om)^{(2)}+(\Gaw\cdot\Gaw)^{(3)}.
\end{align*}
Applying Lemmas \ref{evolutionlemma} and \ref{useful}, \eqref{Gaw3f} and Proposition \ref{propOmc}, we deduce
\begin{align*}
    \|r^2\dkb^3\trchb\|_{2,S}&\les\|r^2\dkb^3\trchb\|_{2,S(-\ub,\ub)}+\int_{-\ub}^u\|r\Om^2(\nab\log\Om)^{(2)}\|_{2,S}+\|r^2\Om^2(\Gaw\cdot\Gaw)^{(3)}\|_{2,S}\\
    &\les\ep_0r+\int_{-\ub}^u\frac{\ep_0 r^{\de+\de_0}}{|u|^{\de+\de_0}}+\frac{\ep^2r^2}{r^{\frac{3}{2}-2\de-2\de_0}|u|^{\frac{1}{2}+2\de+2\de_0}}+r^{3+2\de_0}\int_{-\ub}^u |u|^{-2\de_0}\|\Om\Fb_{1-2\de}^{(2)}\|_{2,S}\\
    &\les\ep_0\frac{r^{1+\de}}{|u|^\de}.
\end{align*}
This concludes the proof of Proposition \ref{prop10.6}.
\end{proof}
\subsection{Estimates for \texorpdfstring{$\mu$}{} and \texorpdfstring{$\mub$}{}}\label{ssecmu}
\begin{prop}\label{propmu}
We have the following estimates:
\begin{align*}
    \|\mu\|_{4,S}&\les\frac{\ep_0}{r^{2-3\de_0}|u|^{3\de_0}},\qquad \|\mu^{(1)}\|_{4,S}\les\frac{\ep_0}{r^{2-\de-2\de_0}|u|^{\de+2\de_0}},\qquad\|\mu^{(2)}\|_{2,S}\les\frac{\ep_0}{r^{2-\de-3\de_0}|u|^{\de+3\de_0}}.
\end{align*}
\end{prop}
\begin{proof}
We recall from Proposition \ref{null}
\begin{align}
\begin{split}\label{mumc4}
\nab_4\mumc+\trch\,\mumc&=r^{-2}\trchc^{(0)}+r^{-1}\rhoc^{(0)}+r^{-1}\Gab\c\Gab+r^{-1}(\Gag\c\Gab)^{(1)}. 
\end{split}
\end{align}
Applying Lemma \ref{evolutionlemma} and Propositions \ref{R0estima}, \ref{prop10.5} and \ref{standardsobolev}, we infer
\begin{align*}
\|r^{2}\mumc\|_{4,S}&\les\|r^{2}\mumc\|_{4,S(u,|u|)}+\int_{|u|}^\ub\|(\trchc,r\rhoc)\|_{4,S}+r\|\Gab\c\Gab\|_{4,S}+\|r(\Gag\c\Gab)^{(1)}\|_{4,S}\\
&\les\ep_0+\int_{|u|}^\ub\frac{\ep_0}{r}+\frac{\ep}{r^{1-\de_0}|u|^{\de_0}}\frac{\ep}{r^{1-\frac{\de_0+3\de_0}{2}}|u|^\frac{3\de_0+\de_0}{2}}+\frac{\ep^2}{r^{1-3\de_0}|u|^{3\de_0}}\\
&\les\ep_0\frac{r^{3\de_0}}{|u|^{3\de_0}}.
\end{align*}
Next, differentiating \eqref{mumc4} by $\dkb$ and applying Proposition \ref{commcor}, we deduce
\begin{align}\label{mumc1}
\nab_4(\dkb\mumc)+\trch(\dkb\mumc)=r^{-2}\trchc^{(1)}+r^{-1}\rhoc^{(1)}+r^{-1}\Gab\c\Gab^{(1)}+r^{-1}(\Gag\c\Gab)^{(2)}.
\end{align}
Applying Lemma \ref{evolutionlemma} and Propositions \ref{R0estima} and \ref{prop10.5}, we obtain
\begin{align*}
\|r^{2}\dkb\mumc\|_{4,S}&\les\|r^{2}\dkb\mumc\|_{4,S(u,|u|)}+\int_{|u|}^\ub \|(\trchc,r\rhoc)^{(1)}\|_{4,S}+\|r(\Gab\c\Gab)^{(1)}\|_{4,S}+\|r(\Gag\c\Gab)^{(2)}\|_{4,S}\\
&\les\ep_0+\int_{|u|}^\ub\frac{\ep_0}{r}+\frac{\ep_0}{r^{1-\de}|u|^{\de}}+\frac{\ep^2}{r^{1-5\de_0}|u|^{5\de_0}}+\frac{\ep^2}{r^{1-\de-2\de_0}|u|^{\de+2\de_0}}\\
&\les\frac{\ep_0}{|u|^{\de+2\de_0}}r^{\de+2\de_0}.
\end{align*}
Finally, differentiating \eqref{mumc1} by $\dkb$ and applying Proposition \ref{commutation}, we obtain
\begin{align*}
\nab_4(\dkb^2\mumc)+\trch(\dkb^2\mumc)=r^{-2}\trchc^{(2)}+r^{-1}\rhoc^{(2)}+r^{-1}(\Gab\c\Gab)^{(2)}+r^{-1}(\Gag\c\Gab)^{(3)}.
\end{align*}
Applying Lemmas \ref{evolutionlemma} and \ref{useful}, \eqref{rhocE} and Propositions \ref{estrb} and \ref{prop10.5}, we obtain
\begin{align*}
\|r^2\dkb^2\mumc\|_{2,S}&\les\|r^2\dkb^2\mumc\|_{2,S(u,|u|)}+\int_{|u|}^\ub \|(\trchc,r\rhoc)^{(2)}\|_{2,S}+\|r(\Gab\c\Gab)^{(2)}\|_{2,S}+r\|(\Gag\c\Gab)^{(3)}\|_{2,S}\\
&\les\ep_0+\int_{|u|}^\ub\frac{\ep_0}{r^{1-\de}|u|^{\de}}+\frac{\ep^2}{r^{1-\de-3\de_0}|u|^{\de+3\de_0}}+\int_{|u|}^\ub r\|\rhoc^{(2)}\|_{2,S}+\|r\Gab\c\Gag^{(3)}\|_{2,S}\\
&\les\frac{\ep_0}{|u|^{\de+3\de_0}}r^{\de+3\de_0}+\frac{\ep_0}{\ep}\frac{\ep}{|u|^\de}r^{\de}+\frac{\ep^2}{|u|^{\de+2\de_0}}r^{\de+2\de_0}\\
&\les\frac{\ep_0}{|u|^{\de+3\de_0}}r^{\de+3\de_0}.
\end{align*}
Recalling from \eqref{murenor} that
\begin{equation*}
\mumc=\mu-\ov{\mu}+\frac{1}{4}\widecheck{\trch\trchb}=\mu+\ov{\rhoc}+\frac{1}{2r}\widecheck{\trchb}-\frac{1}{2r}\trchc+\Gag\c\Gag,
\end{equation*}
we easily deduce from Propositions \ref{prop10.5} and \ref{prop10.6} that
\begin{align*}
    \|\mu\|_{4,S}\les\frac{\ep_0}{r^{1-3\de_0}|u|^{3\de_0}},\qquad \|\dkb\mu\|_{4,S}\les\frac{\ep_0}{r^{2-\de-2\de_0}|u|^{\de+2\de_0}},\qquad \|\dkb^2\mu\|_{2,S}\les\frac{\ep_0}{r^{2-\de-3\de_0}|u|^{\de+3\de_0}}.
\end{align*}
This concludes the proof of Proposition \ref{propmu}.
\end{proof}
\begin{prop}\label{propmub}
We have the following estimates:
\begin{align*}
\|\mub^{(1)}\|_{4,S}\les\frac{\ep_0}{r^{2}},\qquad\quad
\|\mub^{(2)}\|_{2,S}\les\frac{\ep_0}{r^{2-\de}|u|^{\de}}.
\end{align*}
\end{prop}
\begin{proof}
We recall from Proposition \ref{null}
\begin{align}
\begin{split}\label{mumbc3}
\nab_3\mumbc+\trchb\,\mumbc=r^{-2}\trchbc^{(0)}+r^{-1}\rhoc^{(0)}+r^{-1}(\Gaw\c\Gaw)^{(1)}.
\end{split}
\end{align}
Differentiating \eqref{mumbc3} by $\dkb$ and applying Proposition \ref{commcor}, we deduce
\begin{align}\label{mumbc1}
\nab_3\mumbc^{(1)}+\trchb\,\mumbc^{(1)}=r^{-2}\,\trchbc^{(1)}+r^{-1}\rhoc^{(1)}+r^{-1}(\Gaw\c\Gaw)^{(2)}.
\end{align}
Applying Lemma \ref{evolutionlemma} and Propositions \ref{R1estima} and \ref{prop10.6}, we infer
\begin{align*}
\|r^2\mumbc^{(1)}\|_{4,S}&\les\|r^2\mumbc^{(1)}\|_{4,S(u,|u|)}+\int_{-\ub}^u\Om^2\left(\|(\trchbc,r\rhoc)^{(1)}\|_{4,S}+\|r(\Gaw\c\Gaw)^{(2)}\|_{4,S}\right)\\
&\les\ep_0+\int_{-\ub}^u\frac{\ep_0}{r^{1-\de_0}|u|^{\de_0}}+\frac{\ep_0}{r^{1-\de-\de_0}|u|^{\de+\de_0}}+\frac{\ep^2}{r^{\frac{3}{2}-\de-2\de_0}|u|^{\frac{1}{2}+\de+2\de_0}}\\
&\les\ep_0.
\end{align*}
Next, differentiating \eqref{mumbc1} by $\dkb$ and applying Proposition \ref{commutation}, we have
\begin{align*}
    \nab_3(\dkb^2\mumbc)+\trchb(\dkb^2\mumbc)=r^{-2}\trchbc^{(2)}+r^{-1}\rhoc^{(2)}+r^{-1}(\Gaw\c\Gaw)^{(3)}.
\end{align*}
Applying Lemmas \ref{evolutionlemma} and \ref{useful} and Propositions \ref{estbr} and \ref{prop10.6}, we obtain
\begin{align*}
\|r^2\dkb^2\mumbc\|_{2,S}&\les\|r^2\dkb^2\mumbc\|_{2,S(u,|u|)}+\int_{-\ub}^u \Om^2\left(\|(\trchbc,r\rhoc)^{(2)}\|_{2,S}+\|r(\Gaw\c\Gaw)^{(3)}\|_{2,S}\right)\\
&\les\ep_0+\int_{-\ub}^u\frac{\ep_0}{r^{1-\de-2\de_0}|u|^{\de+2\de_0}}+\Om^2\|r\rhoc^{(2)}\|_{2,S}+\Om^2\|r\Gaw\c\Gaw^{(3)}\|_{2,S}\\
&\les\ep_0+\left(\frac{\ep_0}{\ep}+\ep\right)r^{1+2\de_0}\int_{-\ub}^u|u|^{-2\de_0}\|\Om\Fb_{1-2\de}^{(2)}\|_{2,S}\\
&\les\frac{\ep_0}{|u|^\de}r^\de,
\end{align*}
where we used the facts that $\rhoc\in\frac{\ep_0}{\ep}\Fb_{1-2\de}$ and $r^{-1}\Gaw^{(3)}\in\Fb_{1-2\de}^{(2)}$.
Recalling that
\begin{equation*}
\mumbc=\mub-\ov{\mub}+\frac{1}{4}\widecheck{\trch\trchb}=\mub+\ov{\rhoc}+\frac{1}{2r}\widecheck{\trchb}-\frac{1}{2r}\trchc+\Gag\c\Gag,
\end{equation*}
we have from Propositions \ref{prop10.5} and \ref{prop10.6} that
\begin{align*}
    \|\mub^{(1)}\|_{4,S}\les\frac{\ep_0}{r^2},\qquad \quad\|\dkb^2\mub\|_{2,S}\les\frac{\ep_0}{r^{2-\de}|u|^{\de}}.
\end{align*}
This concludes the proof of Proposition \ref{propmub}.
\end{proof}
\subsection{Estimates for \texorpdfstring{$\eta$}{}, \texorpdfstring{$\etab$}{}, \texorpdfstring{$\hch$}{} and \texorpdfstring{$\hchb$}{}}\label{ssec10.7}
\begin{prop}\label{prop10.7}
We have the following estimates:
\begin{align}
\begin{split}\label{esteta}
\|\eta^{(1)}\|_{4,S}\les\frac{\ep_0}{r^{1-3\de_0}|u|^{3\de_0}},\qquad \|\eta^{(2)}\|_{4,S}\les\frac{\ep_0}{r^{1-\de-2\de_0}|u|^{2\de_0}},\qquad\|\etab^{(2)}\|_{4,S}\les\frac{\ep_0}{r}.
\end{split}
\end{align}
We also have
\begin{equation}\label{etaimproved}
    \|\eta\|_{4,S}\les\frac{\ep_0}{r^{1-\de_0}|u|^{\de_0}}.
\end{equation}
\end{prop}
\begin{proof}
We recall from Proposition \ref{null}
\begin{align*}
    d_1\eta&=(-\mu-\rhoc,\sic), \qquad\quad d_1\etab=(-\mub-\rhoc,\sic).
\end{align*}
Applying Propositions \ref{elliptic2d}, we obtain for $q=1,2$
\begin{align*}
\|\eta^{(q)}\|_{4,S}\les r\|(\mu,\rhoc,\sic)^{(q-1)}\|_{4,S},\qquad\quad \|\etab^{(q)}\|_{4,S}\les r\|(\mub,\rhoc,\sic)^{(q-1)}\|_{4,S}.
\end{align*}
Combining with Propositions \ref{R1estima}, \ref{R0estima}, \ref{propmu} and \ref{propmub}, we obtain \eqref{esteta} as stated.\\ \\
Next, we have from Proposition \ref{nulles}
    \begin{align}\label{n4eta}
        \nab_4\eta+\frac{1}{2}\trch\,\eta=r^{-1}\etab-\b+\Gag\c\Gab.
    \end{align}
Applying Lemma \ref{evolutionlemma}, \eqref{esteta} and Proposition \ref{R0estima}, we obtain
\begin{align*}
\|r\eta\|_{4,S}&\les\|r\eta\|_{4,S(u,|u|)}+\int_{|u|}^\ub\|\etab\|_{4,S}+r\|\b\|_{4,S}+r\|\Gag\c\Gab\|_{4,S}\\
&\les\ep_0+\int_{|u|}^\ub \frac{\ep_0}{r}+\frac{\ep^2}{r^{1-\de_0}|u|^{\de_0}} \\
&\les\ep_0\frac{r^{\de_0}}{|u|^{\de_0}},
\end{align*}
which implies \eqref{etaimproved}. This concludes the proof of Proposition \ref{prop10.7}.
\end{proof}
\begin{prop}\label{prop10.1}
We have the following estimates:
\begin{align*}
\|\hch^{(1)}\|_{4,S}&\les\frac{\ep_0}{r},\qquad\qquad\quad\;\,\,\|\hch^{(2)}\|_{4,S} \les\frac{\ep_0}{r^{1-\de}|u|^\de},\\
\|\hchb^{(1)}\|_{4,S}&\les\frac{\ep_0}{r^{1-\de}|u|^{\de}},\qquad\quad\,\|\hchb^{(2)}\|_{4,S}\les\frac{\ep_0}{r^{\frac{1}{2}-\de-\de_0}|u|^{\frac{1}{2}+\de+\de_0}}.
\end{align*}
\end{prop}
\begin{proof}
We recall the following Codazzi equations from Proposition \ref{null}
\begin{align*}
\sdiv\hch&=\frac{1}{2}\nab\trch-r^{-1}\etab-\b+\Gag\c\Gag,\\ 
\sdiv\hchb&=\frac{1}{2}\nab\trchb-r^{-1}\etab+\bb+\Gag\c\Gaw.
\end{align*}
Applying Propositions \ref{elliptic2d}, \ref{R0estima}, \ref{prop10.5}, \ref{prop10.6} and \ref{esteta}, we obtain
\begin{align*}
\|\hch^{(1)}\|_{4,S}&\les r\|(\nab\trch,r^{-1}\etab,\b)\|_{4,S}+\|r\Gag\c\Gag\|_{4,S}\les\frac{\ep_0}{r},\\
\|\hchb^{(1)}\|_{4,S}&\les r\|(\nab\trchb,r^{-1}\etab,\bb)\|_{4,S}+\|r\Gag\c\Gaw\|_{4,S}\les\frac{\ep_0}{r^{1-\de}|u|^{\de}}.
\end{align*}
We also have from Propositions \ref{elliptic2d}, \ref{R1estima}, \ref{prop10.5}, \ref{prop10.6} and \ref{esteta}
\begin{align*}
    \|\hch^{(2)}\|_{4,S}&\les r\|(\nab\trch,r^{-1}\etab,\b)^{(1)}\|_{4,S}+r\|(\Gag\c\Gag)^{(1)}\|_{4,S}\les\frac{\ep_0}{r^{1-\de}|u|^\de},\\
    \|\hchb^{(2)}\|_{4,S}&\les r\|(\nab\trchb,r^{-1}\etab,\bb)^{(1)}\|_{4,S}+r\|(\Gag\c\Gaw)^{(1)}\|_{4,S}\les\frac{\ep_0}{r^{\frac{1}{2}-\de-\de_0}|u|^{\frac{1}{2}+\de+\de_0}}.
\end{align*}
This concludes the proof of Proposition \ref{prop10.1}.
\end{proof}
\subsection{End of the proof of Theorem \ref{M3}}
\begin{prop}\label{prop10.9}
We have the following estimates:
\begin{align}
\sup_\kk r\left|\trch-\frac{1}{r}\right|\les\ep_0,\qquad \sup_\kk r\left|{\trchb}+\frac{4}{r}\right|\les\ep_0.
\end{align}
\end{prop}
\begin{proof}
Applying Lemma \ref{dint}, we obtain
\begin{equation}\label{e3req}
    \Om^2\nab_3\left(\frac{1}{r}\right)=-r^{-2}\Om^2 e_3(r)=-\frac{1}{2r}\ov{\Om^2\trchb}=-\frac{\Om^2\trchb}{2r}+\frac{\widecheck{\Om^2\trchb}}{2r}.
\end{equation}
We recall from Proposition \ref{nulles}
\begin{align*}
    \Om^2\nab_3\trch+\frac{1}{2}\Om^2\trchb\trch=-2\Om^2\mu+2|\Om\eta|^2.
\end{align*}
Hence, we deduce
\begin{align}
\begin{split}\label{dubW}
    \Om^2\nab_3\left(\trch-\frac{1}{r}\right)+\frac{1}{2}\Om^2\trchb\left(\trch-\frac{1}{r}\right)=\Om^2(\mu,r^{-1}\trchbc)^{(0)}+\Om^2\Gab\c\Gab.
\end{split}
\end{align}
Applying Lemma \ref{evolutionlemma}, we have from Propositions \ref{prop10.3} and \ref{prop10.5}
\begin{align*}
    \left\|r\left(\trch-\frac{1}{r}\right)\right\|_{\infty,S}&\les\left\|r\left(\trch-\frac{1}{r}\right)\right\|_{\infty,S(-\ub,\ub)}+\int_{-\ub}^u\frac{\ep_0}{r^{1-\de_0}|u|^{\de_0}}+\|r\Om^2\Gab\c\Gab\|_{\infty,S}\\
    &\les\ep_0+\int_{|u|}^\ub\frac{\ep_0}{r^{1-\de_0}|u|^{\de_0}}+\frac{\ep^2}{r^{1-2\de-4\de_0}|u|^{2\de+4\de_0}} \\
    &\les\ep_0.
\end{align*}
Similarly, we have from Proposition \ref{nulles} and \eqref{e3req}
\begin{align*}
    \nab_3\left(\trchb+\frac{4}{r}\right)+\frac{1}{2}\trchb\left(\trchb+\frac{4}{r}\right)=r^{-1}\,\widecheck{\trchb}^{(0)}+\Gaw\c\Gaw.
\end{align*}
Applying Lemma \ref{evolutionlemma}, we infer
\begin{align*}
    r\left\|\trchb+\frac{4}{r}\right\|_{\infty,S}&\les r\left\|\trchb+\frac{4}{r}\right\|_{\infty,S(-\ub,\ub)}+\int_{-\ub}^u\Om^2\left(\|\trchbc\|_{\infty,S}+\|r\Gaw\c\Gaw\|_{\infty,S}\right)\\
    &\les\ep_0+\int_{-\ub}^u\frac{\ep_0}{r^{1-\de_0}|u|^{\de_0}}+\frac{\ep^2}{r^{1-2\de-\de_0}|u|^{2\de+\de_0}}\\
    &\les\ep_0.
\end{align*}
This concludes the proof of Proposition \ref{prop10.9}.
\end{proof}
\begin{prop}\label{comparison}
We have the following estimate:
\begin{equation*}
 \sup_\kk \left|1-\frac{\ub-u}{2r}\right|\les\ep_0.
\end{equation*}
\end{prop}
\begin{proof}
    We have from Lemma \ref{dint} and Proposition \ref{prop10.9}
    \begin{align*}
        e_4\left(r-\frac{\ub-u}{2}\right)=\frac{r}{2}\ov{\trch}-\frac{1}{2}=\frac{r}{2}\ov{\left(\trch-\frac{1}{r}\right)}=O(\ep_0).
    \end{align*}
    Applying Lemma \ref{evolutionlemma}, we obtain
    \begin{align*}
        \left\|r-\frac{\ub-u}{2}\right\|_{\infty,S}\les \left\|r-\frac{\ub-u}{2}\right\|_{\infty,S(u,|u|)}+\int_{|u|}^\ub\ep_0du\les\ep_0 r.
    \end{align*}
    This concludes the proof of Proposition \ref{comparison}.
\end{proof}
\begin{prop}\label{roundsphere}
We have the following estimates:
\begin{equation*}
 \left\|\mathbf{K}-\frac{1}{r^2}\right\|_{4,S}\les\frac{\ep_0}{r^2},\qquad\quad\|r\nab\mathbf{K}\|_{4,S}\les\frac{\ep_0}{r^{2-\de}|u|^\de}.
\end{equation*}
\end{prop}
\begin{rk}
    Proposition \ref{roundsphere} implies that the spheres $S=S(u,\ub)$ are only almost round at future null infinity $\II_+$.
\end{rk}
\begin{proof}[Proof of Proposition \ref{roundsphere}]
    We have from \eqref{gauss} and Propositions \ref{R1estima} and \ref{prop10.9}
    \begin{align*}
    \left\|\mathbf{K}-\frac{1}{r^2}\right\|_{4,S}&=\left\|\mathbf{K}+\frac{1}{4}\trch\trchb-\left(\frac{1}{4}\trch\,\trchb+\frac{1}{r^2}\right)\right\|_{4,S}\\
    &\les \|\rhoc\|_{4,S}+r^{-1}\left\|\trch-\frac{1}{r}\right\|_{4,S}+r^{-1}\left\|\trchb+\frac{4}{r}\right\|_{4,S}\\
    &\les\frac{\ep_0}{r^2}.
    \end{align*}
Similarly, we have
\begin{align*}
\|r\nab\mathbf{K}\|_{4,S}\les\|r\nab\rhoc\|_{4,S}+\|\nab\trch\|_{4,S}+\|\nab\trchb\|_{4,S}\les\frac{\ep_0}{r^{2-\de}|u|^\de}.
\end{align*}
This concludes the proof of Proposition \ref{roundsphere}.
\end{proof}
Combining Propositions \ref{R1estima}--\ref{roundsphere}, this concludes the proof of Theorem \ref{M3}.

\end{document}